\documentclass{amsart}

\usepackage{amsmath,amsfonts,amssymb,amsthm,mathtools}

\usepackage{verbatim}
\usepackage{amsmath,amsfonts}
\usepackage[mathscr]{euscript} 
\usepackage{amsthm}
\usepackage{url}

\usepackage{graphicx} 

\usepackage{enumitem} 

\usepackage{hyperref} 
\hypersetup{colorlinks} 

\usepackage[colorinlistoftodos]{todonotes}

\RequirePackage{cleveref}
\usepackage{hypcap}
\hypersetup{colorlinks=true, citecolor=darkblue, linkcolor=darkblue}
\definecolor{darkblue}{rgb}{0.0,0,0.7}
\newcommand{\darkblue}{\color{darkblue}}
\newcommand{\defn}[1]{\emph{\darkblue #1}}

\setlist[enumerate]{
	label=\textnormal{({\roman*})},
	ref={\roman*}}

\makeatletter
\def\th@plain{%
	\thm@notefont{}
	\itshape 
}
\def\th@definition{%
	\thm@notefont{}
	\normalfont 
}
\makeatother

\makeatletter
\def\fdsy@scale{1}
\newcommand\fdsy@mweight@normal{Book}
\newcommand\fdsy@mweight@small{Book}
\newcommand\fdsy@bweight@normal{Medium}
\newcommand\fdsy@bweight@small{Medium}

\DeclareFontFamily{U}{FdSymbolB}{}
\DeclareFontShape{U}{FdSymbolB}{m}{n}{
	<-7.1> s * [\fdsy@scale] FdSymbolB-\fdsy@mweight@small
	<7.1-> s * [\fdsy@scale] FdSymbolB-\fdsy@mweight@normal
}{}
\DeclareFontShape{U}{FdSymbolB}{b}{n}{
	<-7.1> s * [\fdsy@scale] FdSymbolB-\fdsy@bweight@small
	<7.1-> s * [\fdsy@scale] FdSymbolB-\fdsy@bweight@normal
}{}

\makeatother

\DeclareSymbolFont{fdrelations}{U}{FdSymbolB}{m}{n}
\SetSymbolFont{fdrelations}{bold}{U}{FdSymbolB}{b}{n}

\DeclareMathSymbol{\lescc}{\mathrel}{fdrelations}{66}

\newtheorem{thm}{Theorem}[section]
\newtheorem{lemma}[thm]{Lemma}

\newtheorem{cor}[thm]{Corollary}

\newtheorem{conj}[thm]{Conjecture}
\newtheorem{conv}[thm]{Convention}

\theoremstyle{definition}

\newtheorem{rem}[thm]{Remark}
\newtheorem{definition}[thm]{Definition}

\numberwithin{figure}{section}
\numberwithin{equation}{section}

\def\wh{\widehat}
\def\bu{\bullet}
\def\emp{\nothing}

\def\zz{\mathbb Z}
\def\nn{\mathbb N}

\def\pp{\mathbb P}

\def\sm{\smallsetminus}

\def\Ga{R}

\def\ga{\gamma}

\def\de{U}

\def\al{\alpha}

\def\ve{\varepsilon}

\def\vk{\varkappa}

\def\ssu{\subset}

\def\<{\langle}
\def\>{\rangle}

\def\Ups{S}
\def\ups{s}
\def\gar{r}

\def\0{{\mathbf 0}}

\def\nothing{\varnothing}

\def\.{\hskip.06cm}
\def\ts{\hskip.03cm}

\def\nin{\noindent}

\def\SP{{\textsc{\#P}}}

\def\Pb{{\text{\bf P}}}

\def\aF{\textrm{F}}
\def\aFr{\textrm{\em F}}
\def\aG{\textrm{G}}

\def\aH{\textrm{H}}

\def\aK{\textrm{K}}
\def\aKr{\textrm{\em K}}
\def\aR{\textrm{R}}
\def\aRr{\textrm{\em R}}

\def\bF{\textbf{\textit{F}}\hskip-0.03cm{}}

\def\bN{\textbf{\textit{N}}\hskip-0.05cm{}}

\def\bG{\textbf{\textit{G}}\hskip-0.03cm{}}
\def\bH{\textbf{\textit{H}}\hskip-0.03cm{}}

\def\bA{\textbf{\textit{A}}\hskip-0.03cm{}}
\def\bB{\textbf{\textit{B}}\hskip-0.03cm{}}
\def\bC{\textbf{\textit{C}}\hskip-0.03cm{}}

\def\ana{\emph{\textsf{a}}}
\def\bnb{\emph{\textsf{b}}}

\def\dndt{{\textsf{d}}}
\def\Lc{L_\circ}

\def\A{A}
\def\B{B}
\def\C{C}
\def\D{D}
\def\E{E}

\def\aAr{\mathrm{A}}
\def\aBr{\mathrm{B}}
\def\aCr{\mathrm{C}}

\DeclareMathOperator{\ab}{\mathbf{a}} 

\DeclareMathOperator{\bb}{\mathbf{b}}
\DeclareMathOperator{\abr}{a}
\DeclareMathOperator{\bbr}{b}

\DeclareMathOperator{\bAr}{A}
\DeclareMathOperator{\bBr}{B}

\DeclareMathOperator{\Cen}{\textnormal{C}} 
\DeclareMathOperator{\Cendown}{\textnormal{C}_{\textnormal{down}}} 
\DeclareMathOperator{\Cenup}{\textnormal{C}_{\textnormal{up}}} 
\DeclareMathOperator{\eb}{\mathbf{e}} 
\DeclareMathOperator{\down}{\textnormal{down}} 
\DeclareMathOperator{\Ec}{\mathcal{E}} 

\DeclareMathOperator{\Fc}{\mathcal{F}} 
\def\Forbdown{F_{\textnormal{down}}} 
\def\Forbup{F_{\textnormal{up}}} 
\DeclareMathOperator{\gb}{\mathbf{g}} 
\DeclareMathOperator{\Gc}{\mathcal{G}} 
\DeclareMathOperator{\GCP}{\textnormal{GCP}} 
\DeclareMathOperator{\Hc}{\mathcal{H}} 
\DeclareMathOperator{\HCP}{\textnormal{HCP}} 
\DeclareMathOperator{\inc}{\textnormal{inc}} 

\newcommand{\Kc}{\mathcal{K}} 

\DeclareMathOperator{\lessr}{\text{less}}
\newcommand{\n}{n} 
\DeclareMathOperator{\one}{\mathbf{1}} 
\DeclareMathOperator{\qb}{\mathbf{q}} 
\DeclareMathOperator{\qbr}{q} 

\DeclareMathOperator{\Rb}{\mathbb{R}} 

\DeclareMathOperator{\rbb}{\mathbf{r}} 
\DeclareMathOperator{\rbbt}{r} 

\DeclareMathOperator{\Reg}{\textnormal{Reg}} 

\def\precc{\prec}
\def\succc{\succ}

\def\precCP{\lescc}

\DeclareMathOperator{\supp}{\textnormal{supp}} 

\DeclareMathOperator{\ub}{\mathbf{u}} 
\def\rub{\mathrm{u}}

\DeclareMathOperator{\up}{\textnormal{up}} 

\DeclareMathOperator{\vb}{\mathbf{v}} 
\DeclareMathOperator{\vbr}{{v}} 

\DeclareMathOperator{\wb}{\mathbf{w}} 
\DeclareMathOperator{\wbr}{{w}} 

\DeclareMathOperator{\wgt}{\mathtt{wt}} 
\DeclareMathOperator{\Zb}{\mathbb{Z}} 
\DeclareMathOperator{\zero}{\mathbf{0}} 

\DeclareMathOperator{\gbr}{g} 

\def\xsf{x}
\def\Cr{\mathcal{C}}
\def\Yu{{Y^{\<u\>}}}
\def\Yt{{Y^{\<t\>}}}

\def\Yw{{Y^{\<w\>}}}

\def\eone{\textbf{\ts\textrm{e}$_1$}}
\def\etwo{\textbf{\ts\textrm{e}$_2$}}

\title{The cross--product conjecture for width two posets}
\date{\today}

\author{Swee Hong Chan}
\address[Swee Hong Chan]{Department of Mathematics, UCLA,  Los Angeles, CA 90095.}
\email{\texttt{sweehong@math.ucla.edu}}

\author[\ts Igor Pak]{Igor Pak}
\address[Igor Pak]{Department of Mathematics, UCLA,  Los Angeles, CA 90095.}
\email{\texttt{pak@math.ucla.edu}}

\author[\ts Greta Panova]{Greta Panova}
\address[Greta Panova]{Department of Mathematics, USC,  Los Angeles, CA 90089.}
\email{\texttt{gpanova@usc.edu}}

\begin{document}
	

\begin{abstract}
The \emph{cross--product conjecture} (CPC) of Brightwell, Felsner and
Trotter (1995) is a two-parameter quadratic inequality for the number of
linear extensions of a poset $P= (X, \prec)$ with given value differences
on three distinct elements in~$X$.  We give two different proofs
of this inequality for posets of width two.  The first proof is
algebraic and generalizes CPC to a four-parameter family.
The second proof is combinatorial and extends CPC to a $q$-analogue.
Further applications include relationships between CPC and other
poset inequalities, and the equality part
of the~CPC for posets of width two.
\end{abstract}

\keywords{Linear extensions of posets, cross--product conjecture, $1/3$--$2/3$ conjecture,
Stanley inequality, Kahn--Saks inequality, Graham--Yao--Yao inequality,
$XYZ$ inequality,
log--concavity, lattice path, Lindstr\"om--Gessel--Viennot lemma, $q$-analogue.}
\subjclass[2020]{Primary: 05A20, \. Secondary: 05A30, 06A07, 06A11}

\maketitle

\section{Introduction}\label{s:intro}

%
Among combinatorial objects, \emph{linear extensions of posets} occupy
a remarkable middle ground between chaos and order.  Posets themselves
come in a variety of shapes and sizes, with applications to many
different areas of mathematics and other sciences.  Consequently,
 linear extensions can also seem unwieldy, and counting them
is known to be computationally intractable (see~$\S$\ref{ss:finrem-hist}).
And yet, there are many positive results for the number of linear
extensions in some special cases, including product and determinant
formulas, polynomial time dynamic programming and approximation
algorithms via Markov chains.

In this paper, we prove several new inequalities between the numbers
of linear extensions for the important special case of posets of width two.
Notably, we resolve the \emph{cross--product conjecture} in this case
and generalize it.  We also show that this generalization is extremely
powerful as it implies a number of (known) results, thus uniting them
under one roof.

\smallskip

\subsection{Classical poset inequalities} \label{ss:intro-classical}
Throughout the paper, let \ts $P=(X,\prec)$ \ts be a finite poset.
A \defn{linear extension} of $P$ is a bijection \. $L: X \to [n]$, such that
\. $L(x) < L(y)$ \. for all \. $x \prec y$.
Let \ts $\Ec(P)$ \ts be the set of linear extensions of~$P$, and
let \ts $e(P):=|\Ec(P)|$. Much of research in the area is motivated by the following:

\medskip

\begin{conj}[{\rm {\. \defn{$\frac{1}{3}-\frac{2}{3}$~conjecture}~\cite{Kis,Fre}}}]
\label{conj:1323}
In every finite poset \. $P=(X,\prec)$ that is not totally ordered, there are two distinct elements \ts $x,y\in X$,
such that
$$\frac13 \, \le \,
\frac{\bigl|\bigl\{L \in \Ec(P)~:~L(x)<L(y)\bigr\}\bigr|}{e(P)} \, \le \, \frac23\,.
$$\end{conj}

\medskip

While open in full generality, the conjecture is proved in several other
special cases (see~$\S$\ref{ss:finrem-1323}).
Notably, it was proved by Linial~\cite{Lin} for
posets of width two, where the conjecture is tight.
For general posets, a breakthrough was made by Kahn and
Saks~\cite{KS} who showed a slightly weaker \. $\frac3{11}-\frac8{11}$ \. version
of the conjecture by using the following remarkable inequality.

\medskip

\begin{thm}[{\cite[Thm~2.5]{KS}}]\label{t:KS}
Let \ts $x,y\in X$ \ts be distinct elements of a finite poset \ts $P=(X,\prec)$.
Denote by \ts $\aFr(k)$ \ts the number of linear extensions \ts $L\in \Ec(P)$,
such that \ts $L(y)-L(x)=k$.  Then:
\begin{equation}\label{eq:KS-ineq}
\aFr(k)^2 \,\. \ge \,\. \aFr(k-1) \,\. \aFr(k+1) \quad \text{for all} \quad k\. > \. 1\ts.
\end{equation}
\end{thm}

\medskip

In a special case when \ts $x=\wh 0$ \ts is the minimal element, the
\defn{Kahn--Saks inequality}~\eqref{eq:KS-ineq} reduces to the earlier
\emph{Stanley inequality}~\cite[Thm~3.1]{Sta}, see also~$\S$\ref{s:log-Stanley}.
Both Stanley and Kahn--Saks inequalities are proved geometrically, by
using the \emph{Alexandrov--Fenchel inequalities}.

In an effort to improve the constants in the \emph{Kahn--Saks \.
$\frac3{11}-\frac8{11}$ \.  theorem}, Brightwell, Felsner and Trotter
formulated the following \emph{cross--product conjecture} (CPC) generalizing
 Theorem~\ref{t:KS} (see $\S$\ref{ss:finrem-CPC}):

\medskip

\begin{conj}[{\rm \defn{cross--product conjecture}~\cite[Conj.~3.1]{BFT}}]
\label{conj:CP}
Let \ts $x,y,z\in X$ \ts be distinct elements of a finite poset \ts $P=(X,\prec)$.
Denote by \ts \ts $\aFr(k,\ell)$ \ts the number of linear extensions $L\in \Ec(P)$,
such that \ts $L(y)-L(x)=k$ \ts and \ts $L(z)-L(y)=\ell$.
Then:
\begin{equation}\label{eq:CP-ineq}
\aFr(k,\ell) \ \aFr(k+1,\ell+1) \ \le \ \aFr(k,\ell+1) \ \aFr(k+1,\ell)
 \quad \text{for all} \quad k, \ts \ell \. \ge \. 1\ts.
\end{equation}
\end{conj}

\medskip

As a motivation, the authors proved the \defn{cross--product
inequality}~\eqref{eq:CP-ineq} for \ts $k=\ell=1$ \ts \cite[Thm~3.2]{BFT}.
Their proof was based on the classical \emph{Ahlswede--Daykin Four Functions
Theorem} (see e.g.~\cite[$\S$6.1]{AS}).  The authors lamented:  ``something
more powerful seems to be needed'' to prove the general form
of~\eqref{eq:CP-ineq}.

\medskip

\subsection{New results} \label{ss:intro-new}
Here is the central result of this paper:

\medskip

\begin{thm}  \label{t:main}
The Cross--Product Conjecture~\ref{conj:CP} holds for all posets
of width two.
\end{thm}

\medskip

We present two different proofs for this theorem, which both have their
own unique advantages.  The first proof use the technique of \emph{characteristic
matrices} which arise in the forthcoming paper~\cite{CP} by the first two
authors.  Roughly speaking, this approach translates the \emph{dynamic programming}
approach to computing \ts $e(P)$ \ts into the language of matrix multiplication.
This approach is versatile enough to allow extensive computations for all
width two posets.

The CPC-type inequalities translate into nonpositivity of all \ts $2\times 2$ \ts minors
of the matrix \ts $\bF_P = \bigl(\aF(k,\ell)\bigr)$, cf.~$\S$\ref{ss:finrem-total}.
We note that this property is preserved under matrix multiplication
(see $\S$\ref{ss:back-CB}); this observation
turned out to be the key to the otherwise very technical proof.  After a rather
extensive setup, we prove that  matrix \ts $\bF_P$ \. is a product of certain elementary
matrices, which implies Theorem~\ref{t:main}.  Our approach also proves the following
extension of the theorem, and suggests the following conjecture:

\medskip

\begin{conj}[{\rm \defn{generalized cross--product conjecture}}]\label{conj:CP-gen}
Let \ts $x,y,z\in X$ \ts be distinct elements of a finite poset \ts $P=(X,\prec)$.
Denote by \ts \ts $\aFr(k,\ell)$ \ts the number of linear extensions $L\in \Ec(P)$,
such that \ts $L(y)-L(x)=k$ \ts and \ts $L(z)-L(y)=\ell$.
Then:
\begin{equation}\label{eq:CP-ineq-gen}
\aFr(k,\ell) \ \aFr(k+i,\ell+j) \ \le \ \aFr(k,\ell+j) \ \aFr(k+i,\ell)
 \quad \text{for all} \quad i, \ts j, \ts k, \ts \ell \. \ge \. 1\ts.
\end{equation}
\end{conj}


\medskip

\begin{thm}\label{t:main-gen}
The Generalized Cross--Product Conjecture~\ref{conj:CP-gen} holds for all
posets of width two.
\end{thm}

\medskip

\noindent
Note that Conjecture~\ref{conj:CP-gen} \ts contains Conjecture~\ref{conj:CP} \ts
when \ts $i=j=1$ (see also~$\S$\ref{ss:finrem-GenCPC} for more on the relation).
Thus, Theorem~\ref{t:main-gen} \ts contains Theorem~\ref{t:main} \ts in that case.

\medskip

Our second proof is entirely combinatorial and gives a surprising
$q$-analogue of Theorem~\ref{t:main}.  In the notation of the theorem,
fix a partition $P$ into two chains \ts $\Cr_1,\Cr_2 \ssu X$,
where \ts $\Cr_1 \cap \Cr_2=\emp$.
The \defn{weight} of a linear extension $L\in \Ec(P)$ is defined as
\begin{equation}\label{eq:def-weight}
 \wgt(L) \ := \ \sum_{x \in \Cr_1} \.  L(x)\..
\end{equation}
The \ts \defn{$q$-analogue} \ts of \ts $\aF(k,\ell)$ \ts is now defined as:
\begin{equation}\label{eq:def-weight-q-analogue}
\aF_q(k,\ell) \ := \  \sum_{L}  \, q^{\wgt(L)}\.,
\end{equation}
where the summation is over all linear extensions \ts $L\in \Ec(P)$,
such that \ts $L(y)-L(x)=k$ \ts and \ts $L(z)-L(y)=\ell$.  We think of
\ts $\aF_q(k,\ell) \in \nn[\ts q]$ \ts as a polynomial with integer coefficients.
Note that the definitions of
both \ts $\wgt(L)$ \ts and \ts $\aF_q(k,\ell)$ \ts depend on the chain
partition (cf.~$\S$\ref{ss:finrem-q}).

\medskip

\begin{thm}[{\rm\defn{$q$-cross--product inequality}}]
\label{t:main-big-q}
Let \. $P=(X,\prec)$ \. be a finite poset of width two, let $(\Cr_1,\Cr_2)$
be a partition of~$P$ into two chains.  For all distinct elements \ts $x,y,z\in X$,
we have:
\begin{equation}\label{eq:CP-ineq-q}
\aFr_q(k,\ell) \ \aFr_q(k+1,\ell+1) \ \leqslant \ \aFr_q(k,\ell+1) \ \aFr_q(k+1,\ell)
\quad \text{for all} \quad  k, \ts \ell \. \ge \. 1\ts,
\end{equation}
where \ts $\aFr_q(k,\ell)$ \ts is defined in~\eqref{eq:def-weight-q-analogue}, and
the inequality between polynomials is coefficient-wise.
\end{thm}

\medskip

Clearly, by setting \ts $q=1$ \ts we recover Theorem~\ref{t:main}.
Our final application of the lattice path approach is the following
necessary and sufficient condition for equality in~\eqref{eq:CP-ineq} and~\eqref{eq:CP-ineq-q}.

\medskip

\begin{thm}[{\rm \defn{cross--product equality}}]
\label{t:cpc_equality}
Let \. $P=(X,\prec)$ \. be a finite poset of width two, $|X|=n$, and let $(\Cr_1,\Cr_2)$
be a partition of~$P$ into two chains.  Fix distinct elements \ts $x,y,z\in X$, and integers
\ts $k,\ell$, s.t.\ \ts $1\le k, \ell \le n-1$.  Denote by \ts \ts $\aFr(k,\ell)$ \ts
the number of linear extensions $L\in \Ec(P)$,
such that \ts $L(y)-L(x)=k$ \ts and \ts $L(z)-L(y)=\ell$. Then the equality
\begin{equation}\label{eq:cpc_equality}
\aFr(k,\ell) \ \aFr(k+1,\ell+1) \ = \ \aFr(k,\ell+1) \ \aFr(k+1,\ell)
\end{equation}

\nin
holds \ts \emph{\underline{if and only if}} \ts one of the following holds:

\smallskip

{\small \bf (a)} \  $\aFr(k,\ell) \ts = \ts \aFr(k+1,\ell)$ \ and \ $\aFr(k,\ell+1) \ts = \ts \aFr(k+1,\ell+1)$,

\smallskip

{\small \bf (b)} \   $\aFr(k,\ell) \ts = \ts \aFr(k,\ell+1)$ \ and \ $\aFr(k+1,\ell) \ts = \ts \aFr(k+1,\ell+1)$,

\smallskip

{\small \bf (c)} \  $\aFr(k+1,\ell) \ \aFr (k,\ell+1)\. = \ts 0$,

\smallskip

{\small \bf (d)} \ There exists an integer~$\ts m$, s.t.\ \. $L(y) = m$ \. for every \. $L\in \Ec(P)$.

\smallskip

\nin
Moreover, the equality~\eqref{eq:cpc_equality} holds \ts \emph{\underline{if and only if}} \ts
\begin{equation}\label{eq:cpc_equality-q}
\aFr_q(k,\ell) \ \aFr_q(k+1,\ell+1) \ = \ \aFr_q(k,\ell+1) \ \aFr_q(k+1,\ell+1)\ts.
\end{equation}
\end{thm}

\medskip

In other words, the theorem says that the cross--product equality~\eqref{eq:cpc_equality}
can occur only in some degenerate cases when the equality is straightforward.
For example,  item~{\small \bf (c)} \ts says that there are zero terms on both sides
of the equality.  Similarly, item~{\small \bf (d)} \ts says that poset $P$ can be
written as a series composition \. $P' \ast y \ast P''$, where $P'$ is an induced order on
$(m-1)$ elements smaller than~$y$, and $P'$ is an induced order on $(n-m)$ elements
greater than~$y$.  In that case both the LHS and the RHS of~\eqref{eq:cpc_equality}
split into products of four identical terms.

We should mention that Theorem~\ref{t:cpc_equality} is modeled after a
remarkable recent result by
Shenfeld and van Handel \cite[Thm~15.3]{SvH}, which gave an equality
criterion for Stanley's inequality~\eqref{eq:stanley} in the generality
of all finite posets.  We postpone until $\S$\ref{ss:finrem-equality}
further discussion of poset equalities.

\smallskip

Our proof of Theorem~\ref{t:main-big-q} is based on interpreting linear extensions
of width two posets as \emph{lattice paths}, a classical approach recently
employed by the authors in~\cite{CPP1}.  To prove inequalities, we construct
explicit injections in the style of the \emph{Lindstr\"om--Gessel--Viennot}
(LGV) \emph{lemma}, by looking at \emph{first intersections} of certain
lattice paths~\cite{GV}.  Theorem~\ref{t:cpc_equality} is then derived by
careful analysis of these injections.

Now, to prove ``$q$-inequalities'', we observe that the \emph{$q$-statistic}
given by the weight in~\eqref{eq:def-weight}, counts the area below the
corresponding paths, and are preserved under our injections.
We refer to~\cite[Ch.~5]{GJ} for both background on lattice paths,
the LGV lemma, and the $q$-statistics by the area.

\medskip

\subsection{The importance of CPC}\label{ss:intro-importance}
We believe that our Generalized Cross--Product Conjecture~\ref{conj:CP-gen}
should be viewed as a central problem in the area.  Our Theorem~\ref{t:main-gen}
is one justification, but we have other reasons to support this claim.  We show
that Conjecture~\ref{conj:CP-gen-even-more}, which is a minor extension of
Conjecture~\ref{conj:CP-gen},  implies the following classical inequalities
in the area:

\smallskip

$\bu$ \ the Kahn--Saks inequality~\eqref{eq:KS-ineq}, see $\S$\ref{ss:power-KS},

\smallskip

$\bu$ \ the Graham--Yao--Yao inequality~\eqref{eq:GYY}, see $\S$\ref{ss:power-GYY}
(see also~$\S$\ref{ss:finrem-GYY}),

\smallskip

$\bu$ \ the XYZ inequality~\eqref{eq:XYZ} by Shepp, see $\S$\ref{ss:power-XYZ}
(see also~$\S$\ref{ss:finrem-XYZ}).

\smallskip

\nin
Each of these implications is a relatively short probabilistic argument largely
independent of the rest of the paper.  We collect them in Section~\ref{s:power}.

\bigskip

\subsection{Structure of the paper}\label{ss:intro-structure}
We begin with a short Section~\ref{s:background} which covers notation and some
background definitions which we use throughout the paper.  In a warmup
Section~\ref{s:power}, we expound on the importance of the cross--product
conjectures by showing that it implies several known inequalities, see above.

The remaining sections are split into two parts giving the algebraic
proof of Theorem~\ref{t:main-gen} and combinatorial proof of
Theorem~\ref{t:main-big-q}.  Both parts are rather technical and lengthy;
the algebraic part is presented in
Sections~\ref{s:characteristic matrix}--\ref{s:proof of CPC}, while the
combinatorial part is presented in
Sections~\ref{s:lattice}--\ref{s:proof of theorem qCP}.
In Section~\ref{s:equality}, we derive the  equality case (Theorem~\ref{t:cpc_equality}),
using our combinatorial approach.  We conclude with
final remarks and open problems in Section~\ref{s:finrem}.

 \bigskip

\section{Preliminaries}\label{s:background}

\subsection{Basic notation} \label{ss:back-not}
We use \ts $[n]=\{1,\ldots,n\}$, \ts $\nn = \{0,1,2,\ldots\}$, and
\ts $\pp = \{1,2,\ldots\}$. Throughout the paper we use $q$ as a variable.
For polynomials \ts $f, g\in \zz[q]$, we write \ts $f\leqslant g$ \ts if
the difference \ts $(g-f) \in \nn[q]$, i.e.\ if \ts $(g-f)$ \ts
is a polynomial with nonnegative coefficients.  Finally, we use
relation \. ``$\precCP$'' \. for vectors, to indicate a property in
Definition~\ref{def:CP-relation}.  Note the difference between relations
$$
x \ts \preccurlyeq_P y\,, \quad a \ts\le \ts b\,,  \quad f \leqslant \ts g \quad  \text{and} \quad
\vb \ts \precCP \ts\wb\ts,
$$
for posets elements, integers, polynomials and vectors, respectively.

\subsection{Fonts and letters} \label{ss:back-fonts}
We adopt somewhat nonstandard notation for both vectors and matrices.
Most matrices are written in bold, with their integer entries in
Roman font with indices in parentheses.  For example, we use a
matrix \ts $\bA$ \ts with entries \ts $\bAr(i,j)$, $1\le i,j\le n$.
Same goes for vectors: we write \ts $\ab=\bigl(\abr(1),\ldots,\abr(n)\bigr)$.
There are several reasons for this, notably because a lot of action happen
to these entries, and the fact that we need space for subscripts as these
vectors are indexed by posets and their elements.

What makes it more complicated, is that we use the usual English notation
for certain especially simple matrices, such as $S, T, W$, etc., and the
fact that both our vectors and matrices are infinite dimensional.
Everything we do can actually be done in $|X|=n$ dimensions, but fixing
dimension brings a host of other technical and notational problems
that we avoid with our choices.

In the second half of the paper we use small Greek letters to denote the
lattice paths, and  capital English letters to denote
the start and end points of these paths in $\Zb^2$.  The coordinates are denoted
by the corresponding small letters.  So e.g.\ we can have a lattice
path \ts $\ga: \A\to\B$, where \ts $\A=(a_1,a_2)$ \ts and \ts $\B=(b_1,b_2)$.
We also use a nonstandard notation for polynomials, writing e.g.\
\. $\aK_q(\A,\B)$ \. for a $q$-polynomial~$\aK$ which counts certain
paths from~$\A$ to~$\B$.  Finally, we use curvy English letters to denote
sets of path, i.e.\ we would write that \ts $\aK=|\Kc|$ \ts is the number of
lattice paths in the set~$\Kc$.  Note the difference in fonts in all
these cases.

\medskip

\subsection{Posets} \label{ss:back-def}
Let \. $P=(X,\precc_P)$ \. be a finite poset with ground set $X$ of size~$n$.
We write $\prec$ in place of $\precc_P$ whenever the underlying poset is clear.
For every $x \in X$, denote
$$
\lessr_P(x) \ := \ \bigl|\{\ts y\in X~:~y\prec x\ts \}\bigr| \qquad \text{and}\qquad
\inc_P(x) \ := \ \bigl|\{\ts y\in X~:~y\ne x, \ y\nprec x \ \text{and} \ y \nsucc x\ts\}\bigr|
$$
the numbers of poset elements that are strictly smaller and that are incomparable to~$x$,
respectively.

A \defn{linear extension} of $P$ is a bijection \. $L: X \to [n]$, such that
\. $L(x) < L(y)$ \. for all \. $x \prec y$.
Denote by \ts $\Ec(P)$ \ts the set of linear extensions of $P$,
and write \. $e(P):=|\Ec(P)|$.  For a subset $Y\ssu X$ and a poset $P=(X,\prec_P)$,
define a \defn{restriction} $P'=P|_Y$ to be a poset $P':=(Y,\prec_Y)$ with the
order $\prec_Y$ induced by~$\prec_P$.  Similarly, a for a linear extension $L\in \Ec(P)$,
define a \defn{restriction} $L'=L|_Y \in \Ec(P')$, with the linear order on~$Y$ induced
by the linear order on~$L$.

\medskip

\subsection{Correlation matrix} \label{ss:back-matrix}
Fix three distinct elements $z_1,z_2,z_3$ of $X$ throughout this paper.
For every \ts $i,j \geq 1$, denote by \ts $\Fc(i,j)$ \ts the set of linear
extensions of $X$ defined as
\begin{equation}\label{eq:Fmatrix}
	\Fc(i,j) \ := \ \bigl\{\. L \in \Ec(X) \. \mid \.  L(z_2)-L(z_1) =i\., \ L(z_3)-L(z_2) = j \. \bigr\}\..
\end{equation}
Let \. $\aF(i,j) \. := \. \bigl|\Fc(i,j)\bigr|$, for all \ts $i,j \geq 1$.

Denote by \ts $\bF=\bF_P$ \ts the \ts $\pp \times \pp$ \ts matrix with integer
entries \ts $\aF(i,j)$.  We call it the \defn{correlation matrix} of poset~$P$.
While this matrix has a bounded support for all finite posets,
for technical reasons it is convenient to keep it infinite.
We do the same for the \defn{$q$-correlation matrix} \ts $\bF_q=\bF_{q,P}$ \ts with
polynomial entries $\aF_q(i,j) \in \nn[q]$ defined as in the introduction:
\[
\aF_q(i,j) \ := \  \sum_{L \in \Fc(i,j)}  \. q^{\wgt(L)} \qquad \text{for all \ $i, \ts j \. \geq \. 1$}\ts.
\]

\medskip

\subsection{Cross--product inequalities} \label{ss:back-ineq}
We can now restate the inequalities in the new notation. First,
the \defn{cross--product inequality}~\eqref{eq:CP-ineq} can be written concisely
in the matrix form:
\begin{equation}\label{eq:CPconjecture}
\det \.
\begin{bmatrix}
		\aF(i,j) & \aF(i,j+1)\\
		\aF(i+1,j) & \aF(i+1,j+1)
\end{bmatrix}
\ \leq \ 0 \qquad \text{for all \ \, $i, \, j \. \ge 1$\ts.}
\end{equation}
Similarly, the \defn{generalized cross--product inequality}~\eqref{eq:CP-ineq-gen}
can be written as:
\begin{equation}\label{eq:CPconjecture-gen}
\det \.
\begin{bmatrix}
		\aF(i,j) & \aF(i,\ell)\\
		\aF(k,j) & \aF(k,\ell)
\end{bmatrix}
\ \leq \ 0 \qquad \text{for all \ \, $1 \leq i \leq k, \ 1 \leq j \leq \ell$\ts.}
\end{equation}
This is the form in which we prove these inequalities for posets of width two.

\smallskip

Note that for the purposes of these inequalities, without loss of generality
we can always assume that elements $z_1, z_2,z_3$ satisfy
\begin{equation}\label{eq:z1z2z3}
	z_1 \ \precc_P \ z_2 \ \precc_P \ z_3\..
\end{equation}
Indeed, since \. $i,j,k,\ell \geq 1$, all the linear extensions \ts $L\in \Ec(P)$ \ts
counted by $\aF(i,j)$, $\aF(i,\ell)$, $\aF(k,j)$ and $\aF(k,\ell)$,
satisfy \. $L(z_1) < L(z_2) < L(z_3)$.  Thus the ordering in~\eqref{eq:z1z2z3}
can always be added to~$\prec_P$\ts.

\medskip

\subsection{Cauchy--Binet formula} \label{ss:back-CB}
Below we rewrite the \defn{Cauchy--Binet formula} for \ts $2\times 2$ \ts minors
in our matrix notation.
For every three \ts $n\times n$ \ts matrices \ts $\bA=\bB\ts\bC$, we have:
\begin{equation}\label{eq:CB-formula}
	   \det \begin{bmatrix}
		\aAr(i,j) & \aAr(i,\ell)\\
		\aAr(k,j) & \aAr(k,\ell)
	\end{bmatrix} \ = \  \sum_{1 \.\leq t\. \leq\. m\. \le\. n} \
\det \begin{bmatrix}
	\aBr(i,t) & \aBr(i,m)\\
	\aBr(k,t) & \aBr(k,m)
\end{bmatrix} \, \det \begin{bmatrix}
\aCr(t,j) & \aCr(t,\ell)\\
\aCr(m,j) & \aCr(m,\ell)
\end{bmatrix}\ts,
\end{equation}
for all \ts $1\le i \leq k\le n$ \ts and \ts $1\le j \leq \ell\le n$.
In particular, when both \ts $\bB$ \. and \ts $\bC$ \. have
nonnegative \ts $2\times 2$ \ts minors, the so does~$\bA$.
This simple property will be used several times in the algebraic proof.

\medskip

\subsection{Posets of width two} \label{ss:back-width-2}
\defn{Width} of a poset is the size of the \emph{maximal antichain}.
Unless stated otherwise, we assume that all posets have width two.
By the \emph{Dilworth theorem}, every poset \ts $P=(X,\prec)$ \ts of width two
can be partitioned into two chains.
From this point on, without loss of generality, we fix a partition of~$P$
into chains \. $\Cr_1,\Cr_2 \ssu X$~:
\[
\Cr_1 \ := \ \{ \. \alpha_1 \prec \ldots \prec \alpha_{\ana}\. \}\.,
\ \quad  \Cr_2 \ := \  \{ \. \beta_1 \prec \ldots \prec \beta_{\bnb} \. \}\.,
\ \quad \text{for some \ $\ana+\bnb =n$\.,}
\]
where \. $\Cr_1 \cup \Cr_2=X$ \. and \. $\Cr_1 \cap \Cr_2=\emp$.
The \defn{weight} of a linear extension \ts $L\in \Ec(P)$ \ts can then be written as:
\begin{equation}\label{eq:definition weight L}
 \wgt(L) \ = \  \sum_{x \in \Cr_1} \. L(x) \ = \ \sum_{j=1}^{\ana} \. L(\al_j)\..
\end{equation}
We will use this notation throughout the paper.

\bigskip

\section{The power of CPC}\label{s:power}

\smallskip

In this short section we show the power of the Cross--Product Conjecture by deriving
three earlier results directly from it: the Kahn--Saks inequality (Theorem~\ref{t:KS}),
the Graham--Yao--Yao inequality (Theorem~\ref{t:GYY}),
and the \ts $XYZ$ \ts inequality (Theorem~\ref{t:xyz}).

\subsection{Conjecture~\ref{conj:CP} implies Theorem~\ref{t:KS}} \label{ss:power-KS}
Let \ts $P=(X,\prec)$ \ts be a finite poset, let $|X|=n$, and let \ts $x, z \in X$.
Denote by \ts $Q=(X',\prec')$ \ts be a poset on a set \ts $X'=X+y$,
with added element \ts $y$ \ts incomparable with~$X$ in the order~$\prec'$.

We compare the Kahn--Saks inequality~\eqref{eq:KS-ineq} for the poset~$P$ with
and the cross--product inequality~\eqref{eq:CP-ineq} for the poset~$Q$.
Expounding on the notation in the introduction, denote
$$
\aligned
\aF_P(k\ts;\ts x,z) \, & := \, \bigl|\{\ts L\in \Ec(P)~:~L(z)-L(x)=k\ts\}\bigr|\ts, \\
\aF_Q(i,j\ts;\ts x,y,z) \, & := \, \bigl|\{\ts L\in \Ec(Q)~:~L(z)-L(y)=i, \. L(y)-L(x)=j\ts\}\bigr|\ts.
\endaligned
$$
Observe that in the construction above, we have:
$$
\aF_P(k+\ell-1\ts;\ts x,z) \, = \, \aF_Q(k,\ell\ts;\ts x,y,z) \quad \text{for all \ \. $k,\ell\ge 1$}\ts.
$$
Indeed, the only constraint on $L(y)$ in the RHS is the difference with $L(x)$ and $L(z)$.
Since \ts $|X'|=n+1$, the restriction of $L\in \Ec(Q)$ to $X$ give the bijection.

Now, the cross--product inequality~\eqref{eq:CP-ineq} gives:
$$
\aF_Q(k+1,\ell\ts;\ts x,y,z) \ \aF_Q(k,\ell+1\ts;\ts x,y,z) \ \ge \
\aF_Q(k,\ell\ts;\ts x,y,z) \ \aF_Q(k+1,\ell+1\ts;\ts x,y,z)\ts.
$$
This translates into
$$
\aF_P(k+\ell\ts;\ts x,z)^2 \ \ge \ \aF_P(k+\ell-1\ts;\ts x,z) \ \aF_P(k+\ell+1\ts;\ts x,z)\ts,
$$
which is the desired Kahn--Saks inequality~\eqref{eq:KS-ineq}. \qed

\smallskip

\begin{rem}{\rm
Note that this reduction increases the width of the poset.  Thus,
the cross--product inequality for posets of width two does not imply
anything about the Kahn--Saks inequality by this argument.  We do, however,
prove the $q$-Kahn--Saks inequality for posets of width two in a
followup paper, see~$\S$\ref{ss:finrem-KS}.
}\end{rem}

\medskip

\subsection{The (even more) generalized cross--product inequality} \label{ss:power-CPC-even-more}
From the point of view of this paper, it is best to state the
Generalized Cross--Product Conjecture~\ref{conj:CP-gen} in an even more
general form:

\smallskip

\begin{conj}\label{conj:CP-gen-even-more}
In conditions of Conjecture~\ref{conj:CP-gen}, we have:
\begin{equation}\label{eq:CP-ineq-gen-even-more}
\aFr(i,j) \ \aFr(k,\ell) \ \le \ \aFr(i,\ell) \ \aFr(k,j)
 \quad \text{for all} \quad i \le k, \ts j \le \ell\ts.
\end{equation}
\end{conj}

\smallskip

Substantively, the only difference is that in notation of Conjecture~\ref{conj:CP-gen}
we now allow integers \ts $i$ \ts and \ts $j$ \ts to be negative.  This corresponds to
changing the relative order of elements $z_1,z_2,z_3$ in~\eqref{eq:z1z2z3}.
While this makes a large number of (easy) change of sign implications, the proof
of this conjecture for posets of width two follows verbatim.

\smallskip

\begin{thm}\label{t:main-gen-even-more}
Conjecture~\ref{conj:CP-gen-even-more} holds for
posets of width two.
\end{thm}

\smallskip

Fix \ts $x, z\in X$ \ts and define \. $\aR(i,j):=\aR_P(i,j)$ \. as follows:
\begin{align*}
	\aR(i,j) \ := \ \big| \{\. L \in \Ec(P) \,~:~\,  L(x) =i\., \ L(z)= j \. \}\big|
\qquad \text{for all \ \. $i, \ts j \in \nn$.}
\end{align*}

\smallskip

\begin{cor}\label{c:CPC}  In notation above, we have:
	\[
	\aRr(i,j) \ \aRr(k,\ell)  \ \geq   \ \aRr(i,\ell) \ \aRr(k,j) \quad
        \text{for all \ \ $i \leq k$ \ and \ $j \leq \ell$.}
	\]
\end{cor}

\begin{proof}
This inequality follows immediately from Theorem~\ref{t:main-gen-even-more},
by setting $x\gets x$, $y\gets\wh 0$, and $z\gets z$, where \ts $\wh 0$ \ts
is a global minimal element added to~$P$.  The details are
straightforward.
\end{proof}

\medskip

\subsection{GYY inequality} \label{ss:power-GYY}
For the rest of this section we use a probabilistic language
on the set  $\Ec(P)$ of linear extensions of~$P$.

An \defn{event} is a subset of $\Ec(P)$.
A \defn{forward atomic event} is an event that is of the form
\[
\bigl\{\. L \in \Ec(P) \,~:~\, L(\alpha_i) < L(\beta_j) \.  \bigr\},
\]
for some \ts $\alpha_i \in \Cr_1$ \ts and \ts $\beta_j \in \Cr_2$.
A \defn{forward event}~$A$ is an intersection \. $A_1 \cap \ldots \cap A_k$ \.
of forward atomic events $A_1,\ldots, A_k$.  We denote by \ts $\Pb := \Pb_{P}$ \ts
the uniform measure on linear extensions of~$P$.

\smallskip

\begin{thm}[{\cite[Thm~1]{GYY}}]\label{t:GYY}
	Let $P$ be a finite poset of width two, and let $A$ and
$B$ be forward events. Then:
	\begin{equation}\label{eq:GYY}
	\Pb[A \cap  B] \ \geq \  \Pb[A] \, \Pb[B]\ts.
	\end{equation}
\end{thm}

\smallskip

The theorem was originally proved by Graham, Yao and Yao in~\cite{GYY} using
a lattice paths argument, and soon after reproved by Shepp~\cite{She}
using the \emph{FKG inequality}.  We refer to~\eqref{eq:GYY} as the
\defn{Graham--Yao--Yao {\rm (GYY)} inequality}.

\smallskip

Below we rederive the GYY inequality first for atomic, and then for general
forward events.  The aim is to give an elementary self-contained proof of
Theorem~\ref{t:GYY}.

\medskip

\subsection{CPC implies GYY inequality} \label{ss:power-CPC-GYY}
We start with the following lemma:

\smallskip

\begin{lemma}\label{l:GYY atomic}
GYY inequality~\eqref{eq:GYY} holds for atomic forward events.
\end{lemma}


\begin{proof}
Let
\[
A \ = \  \bigl\{\.L \in \Ec(P) \,~:~\, L(\alpha_{r}) < L(\beta_{s})\. \bigr\},
\qquad B \  = \  \bigl\{\. L \in \Ec(P) \,~:~\,  L(\alpha_{t}) < L(\beta_{u}) \.\bigr\},
\]
where \ts  $r,t \in [\ana]$ \ts and \ts $s,u \in [\bnb]$.

Suppose \ts $L\in A$.  Then $L$ satisfies \.
$L(\alpha_r) < L(\alpha_{r+1})< \ldots < L(\alpha_{\ana})$ \.
and \. $L(\alpha_r) < L(\beta_s)< \ldots < L(\beta_{\bnb})$.
This implies that \ts $L(\alpha_r)< r+s$.  In the opposite direction,
for every \ts $L\in \Ec(P)$, \ts $L(\alpha_r)< r+s$, we have \ts $L\in A$.
We conclude:
	\begin{align*}
		A \ &= \  \bigl\{\. L \in \Ec(P) \,~:~\, L(\alpha_{r}) < r+s \. \bigr\}, \qquad
        B \ = \   \bigl\{\. L \in \Ec(P) \,~:~\,  L(\alpha_{t}) < t+u \.  \bigr\}\ts.
	\end{align*}
	Now let \ts $x=\alpha_{r}$, \ts $z = \alpha_{t}$, and let $L\in \Ec(P)$ be a uniform random linear extension of~$P$.
	Write \ts $c_1:= r+s$ \ts and \ts $c_2:=t+u$.
	Under this notation, we have:
	\begin{align*}
		\Pb[A \cap B]  \ &= \, \Pb\bigl[L(\alpha_{r})<r+s, \ts L(\alpha_{t})<t+u \bigr] \
        = \sum_{i < c_1, \ts j < c_2 } \, \Pb\bigl[L(x)=i, \ts L(z)=j\bigr] \ = \sum_{i < c_1, j < c_2 } \frac{\aR(i,j)}{e(P)}\..
	\end{align*}
By the same reasoning, we have:
	\begin{align*}
	\Pb[A \cap B^c]  \ &=  \    \sum_{i < c_1, \ell \geq  c_2 } \frac{\aR(i,\ell)}{e(P)}\ts, \qquad \Pb[A ^c \cap B]  \
    = \sum_{k \geq  c_1, j <  c_2 } \frac{\aR(k,j)}{e(P)} \ts, \qquad
	\Pb[A^c \cap B^c]  \ =  \sum_{k \geq  c_1, \ell  \geq   c_2 } \frac{\aR(k,\ell)}{e(P)}\ts.
	\end{align*}
	It then follows from these equations that
	\begin{align*}
		\Pb[A \cap B] \ \Pb[A^c \cap B^c] \ = \ \sum_{i < c_1, j < c_2 } \, \sum_{k \geq  c_1, \. \ell \geq   c_2 }
        \, \frac{\aR(i,j) \, \aR(k,\ell)}{e(P)^2}\,,\\
		\Pb[A  \cap B^c] \ \Pb[A^c \cap B]  \ = \ \sum_{i < c_1, j < c_2 } \, \sum_{k \geq  c_1, \. \ell \geq   c_2 }
        \. \frac{\aR(i,\ell) \, \aR(k,j)}{e(P)^2}\,.
	\end{align*}
	Note that in the two equations above, we have $i \leq k$ and $j \leq \ell$.
	It then follows from Corollary~\ref{c:CPC} that
	\begin{equation}\label{eq:atomic 1}
		\Pb[A \cap B] \ \Pb[A^c \cap B^c] \ \geq \
		\Pb[A \cap B^c] \ \Pb[A^c \cap B]\..
	\end{equation}
	On the other hand, by the inclusion exclusion we have:
	\begin{equation}\label{eq:atomic 2}
	\begin{split}
		\Pb[A \cap B] \ \Pb[A^c \cap B^c] \ &= \ \Pb[A \cap B] \ - \ \Pb[A\cap B] \, \Pb[A \cup B]\ts,\\
		\Pb[A \cap B^c] \ \Pb[A^c \cap B] \  &= \  \bigl(\Pb[A] -  \Pb[A\cap B]\bigr) \, \bigl(\Pb[B] -  \Pb[A\cap B]\bigr) \\
		&= \  \Pb[A] \ \Pb[B] \ - \ \Pb[A \cap B] \ \Pb[A \cup B]\..
	\end{split}
	\end{equation}
	The lemma now follows by combining~\eqref{eq:atomic 1} and~\eqref{eq:atomic 2}.
\end{proof}

\smallskip

\begin{proof}[Proof of Theorem~\ref{t:GYY}]
Let \. $A=A_1\cap \ldots \cap A_k$ \. and \. $B=B_1 \cap \ldots \cap B_\ell$ \. be
forward events, where \. $A_1,\ldots, A_k$ \. and \. $B_1,\ldots, B_\ell$ \. are
forward atomic events.  We prove the theorem by induction on \ts $k+\ell$.
The base of induction \ts $k=\ell=1$ \ts is given in Lemma~\ref{l:GYY atomic}.

For \ts $\ell>1$, let \. $C:=B_1 \cap \ldots \cap B_{\ell-1}$ \. and \. $D:=B_\ell$.
Without loss of generality, assume that \ts $\Pb[C] >0$, as otherwise \.
$\Pb[B] \leq \Pb[C]=0$ and \eqref{eq:GYY} is trivially true.
Note that
	\begin{align*}
		\Pb[A \cap B] \ = \  \Pb[A \cap C \cap D]
		\ = \ \Pb\big[A \cap D \mid C\big] \, \Pb[C]\ts.
	\end{align*}

Now let \ts $P':=(X,\prec')$ \ts be the poset for which the relation \. $\prec'$ \. is defined
by~$C$.  Formally, we have \. $x \prec' y$ \. if and only if $L(x)<L(y)$ \. for all $L \in C$.
Since \ts $\Pb[C] >0$, poset~$P'$ is well defined.  Clearly, $\Ec(P') \subseteq \Ec(P)$.

Write \. $\Pb':= \Pb_{P'}$ \. for the uniform measure on~$\Ec(P')$.
Note that the probability measure $\Pb'[\ts H \ts ]$ is equal to
the conditional probability measure \.
$\Pb[\ts H \ts \mid \ts C\ts ]$, for all \. $H \subseteq \Ec(P')$.
It then follows that
\begin{align*}
	\Pb\big[A \cap D \mid C\big] \ \Pb[C] \ \ = \ \   \Pb'[A \cap D ] \ \Pb[C] \ \  \geq  \ \  \Pb'[A] \ \Pb'[D] \ \Pb[C]\.,
\end{align*}
where the last inequality is by applying \eqref{eq:GYY} to the event \ts $A$ \ts and \ts $D$ \ts on the poset~$P'$.
Rewriting the right side of the equation above in terms of the measure \ts $\Pb$, we obtain:
\begin{align*}
	 \Pb'[A] \ \Pb'[D ] \ \Pb[C]
	\ \ & = \ \ \Pb\big[A \mid C\big] \ \Pb\big[D \mid C\big] \ \Pb[C] \ \
= \ \ \Pb\big[A \mid C\big] \ \Pb[D \cap C] \\
	&=  \ \  \Pb[A \mid C] \ \Pb[B] \quad \ \geq \quad
	\Pb\big[A\big] \ \Pb[B]\.,
\end{align*}
where the last inequality is by applying \eqref{eq:GYY} to the events \ts $A$ and \ts $C$ \ts on the poset~$P$.
The case \ts $k>1$ \ts follows analogously.
\end{proof}

\medskip

\subsection{XYZ inequality} \label{ss:power-XYZ}
This following remarkable inequality is saying that there is a positive correlation
on random linear orders of events recording partial information.

\medskip

\begin{thm}[{\rm \defn{$XYZ$ inequality}, {\rm Shepp~\cite{She-XYZ}}}]
\label{t:xyz}
Let \ts $x,y,z\in X$ \ts be distinct elements of a finite poset \ts $P=(X,\prec)$.  Then:
\begin{equation}\label{eq:XYZ}
\Pb\big[L(x)<L(y), \ts L(x) < L(z)\big] \  \ \geq \ \ \Pb\big[L(x)<L(y)\big] \ \Pb\big[L(x) < L(z)\big]\..
\end{equation}
\end{thm}

\medskip

We show that it follows from the (unproven) Generalized Cross--Product
Conjecture.

\medskip

\begin{thm}\label{t:CPC-XYZ}
Conjecture~\ref{conj:CP-gen-even-more} implies Theorem~\ref{t:xyz}.
\end{thm}

\smallskip

\begin{proof}
To avoid the clash of notation, we will prove the ``$uvw$ inequality'' instead:
\begin{equation*}\label{eq:uvw}
\Pb\big[L(u)<L(v), \ts L(u) < L(w)\big] \  \ \geq \ \ \Pb\big[L(u)<L(v)\big] \ \Pb\big[L(u) < L(w)\big]\..
\end{equation*}
Let \ts $A, B \subseteq\Ec(P)$ \ts given by
	\[
        A \ := \ \bigl\{\. L \in \Ec(P)  \. : \. L(u) < L(v)\. \bigr\}, \qquad
        B \ := \ \big\{\. L \in \Ec(P)  \. : \. L(u) < L(w)\. \big\}\ts.
    \]
The theorem can then be restated as
$$
\Pb[A \cap B] \  \geq \  \Pb[A] \ \Pb[B]\ts.
$$
In the notation of Conjecture~\ref{conj:CP-gen}, let \ts $x\gets v$, \ts $y\gets u$ \ts and \ts $z\gets w$.
Then we have:
	\begin{align*}
		\Pb[A \cap B]  \ &= \ \Pb\bigl[L(y)< L(x), L(y) < L(z)\bigr]
		\ = \    \sum_{i < 0, \, \ell >0 } \Pb\bigl[L(y)-L(x)=i, L(z)-L(y)=\ell\bigr] \\
		& = \ \sum_{i < 0, \,  \ell >0 } \frac{\aF(i,\ell)}{e(P)}\,.
	\end{align*}
	 By the same  reasoning, we have
	 	\begin{align*}
	 	\Pb[A \cap B^c]  \ =    \sum_{i < 0, \, j <0 } \frac{\aF(i,j)}{e(P)}\,, \quad \Pb[A^c \cap B]  \ =  \sum_{k >0, \, \ell >0 } \frac{\aF(k,\ell)}{e(P)} \,, \quad
	 	\Pb[A^c \cap B^c]  \ =   \sum_{k >0, \, j <0 } \frac{\aF(k,j)}{e(P)}\,.
	 \end{align*}
 It then follows from these equations that
 \begin{equation}\label{eq:xyzdifference}
 	\Pb[A \cap B] \ \Pb[A^c \cap B^c] \ - \ \Pb[A \cap B^c] \ \Pb[A^c \cap B] \ = \  \sum_{\substack{i < 0, \, j <0 \\ k >0, \, \ell >0}} \frac{\aF(i,\ell)\ \aF(k,j) \ - \  \aF(i,j)\ \aF(k,\ell)}{e(P)^2}\,.
 \end{equation}
Now note that the right side \eqref{eq:xyzdifference} is a sum of nonnegative
terms by Conjecture~\ref{conj:CP-gen-even-more}.
The rest of the proof follows verbatim the proof of Lemma~\ref{l:GYY atomic}
given above.  The minor changes in the summation ranges are straightforward.
\end{proof}

\bigskip

\section{Characteristic matrices}\label{s:characteristic matrix}

In this section we convert the basic dynamic programming approach to
computing the number of linear extensions of posets of width two into
an algebraic statement as a matrix product of certain characteristic
matrices.  These matrices will be further analyzed in the next section.

\medskip

\subsection{Recursion formula}\label{ss:char-recursion}

Let $P$ be a finite poset of width two.
Denote by  \ts $\bN_P$  \ts the \ts $\pp \times \pp$ \ts matrix with entries
\begin{equation*}\label{eqN}
	N_P(i,j) \ :=  \ \bigl| \bigl\{ L\in \Ec(P)~:~L(\beta_1)=i\,, \ L(\beta_{\bnb}) =  j+\lessr_P(\beta_{\bnb})\bigr\}\bigr|\..
\end{equation*}
%
%
Let $\xsf_1$ be the element of $X$ given by
\begin{align*}
	\xsf_1 \ := \
	\begin{cases}
		\, \alpha_1 \, & \ \text{ if \ } \alpha_1  \. \precc_P \. \beta_1\,,\\
		\, \beta_1 \, & \ \text{ otherwise\ts.}
	\end{cases}
\end{align*}
Denote \ts $X':=X- \{\xsf_1\}$, and let \ts $P'=(X',\prec)$ \ts be the induced subposet.

\medskip

\begin{lemma}  \label{l:recursion}
Let \ts $i,j \geq 1$.
If $\xsf_1= \alpha_1$, then  we have
\begin{align}\label{eq:recursion NP 1}
	N_P(i,j) \quad = \quad
	\begin{cases}
		0 & \text{ if } \ i=1\ts,\\
		N_{P'}(i-1,j) & \text{ if } \ i >1\ts.
	\end{cases}
\end{align}
If $\xsf_1= \beta_1$, then we have
\begin{align}\label{eq:recursion NP 2}
N_P(i,j) \quad = \quad
\begin{cases}
	\sum_{k=i}^\infty N_{P'}(k,j)  & \text{ if } \ i\leq \inc_P(\xsf_1)+1,\\
	0 & \text{ if } \ i > \inc_P(\xsf_1)+1\..
\end{cases}
\end{align}
\end{lemma}

\smallskip

\begin{proof}
We associate to each linear extension \ts $L\in \Ec(P)$ \ts
a restriction \ts $L'\in \Ec(P')$ defined as in~$\S$\ref{ss:back-def}.
Note that this map $\phi: \Ec(P) \to \Ec(P')$ is a surjection, since for every \ts $L'\in \Ec(P')$ \ts
we can always set \ts $L(x_1):=1$, $L(y):= L'(y)$ for all $y\ne x_1$.

There are two  possibilities.
First, if $\xsf_1 =\alpha_1$, then the map~$\phi$ is a bijection. This follows from \ts
$L(\xsf_1)=1$ \ts for every $L \in \Ec(X)$, and this implies~\eqref{eq:recursion NP 1}, as desired.

Second, if $\xsf_1 =\beta_1$, let $\ell:=\inc_P(\beta_1)+1$.
Then every linear extension \ts $L \in \Ec(P)$ \ts satisfies
$L(\beta_1) < L(\alpha_\ell)$.
This implies that, every  $x \in X$ satisfying  $L(x) < L(\beta_1)$,
is contained in $\{\alpha_1,\ldots, \alpha_{\ell-1} \}$.
This in turn implies that $L(\beta_1) \leq \ell$.
We then  conclude that \. $N_P(i,j)=0$ \. if \. $i= L(\beta_1) > \ell$\.,
which proves the second part of~\eqref{eq:recursion NP 2}.

Now suppose that $\xsf_1 =\beta_1$ and $i \leq \ell$.
Let $L' \in \Ec(P')$, and let $k:=L'(\beta_2)$.
Then  every linear extension \ts $L \in \Ec(P)$ \ts such that \ts $L\in \phi^{-1}(L')$ \ts
satisfies  \. $L(\beta_1)<L(\beta_2) = k+1$.  In fact, if \ts $i < k+1$,
then \ts $\phi^{-1}(L')$ \ts contains a linear extension \ts $L \in \Ec(P)$ \ts such that \ts $L(\beta_1)=i$.
Indeed, this is the unique linear extension \ts $L \in \Ec(P)$ \ts for which \.
$L(\alpha_{i-1}) <  L(\beta_1) < L(\alpha_i)$ \. and \ts $L|_{X'}=L'$.
Hence we have \. $N_P(i,j) = \sum_{k=i}^\infty N_{P'}(k,j)$\.,
which proves the first part of~\eqref{eq:recursion NP 2}.
This completes the proof of the lemma.
\end{proof}

\medskip

\subsection{Main definitions}\label{ss:char-main-def}
Define the \defn{minimal linear extension} $\Lc$ of $P$ to be  the unique linear
extension of~$P$, such that \ts $\Lc(x) \leq  \Lc(y)$ \ts if \ts $x \prec y$, and \ts
$\Lc(y) \leq  \Lc(x)$ \ts if \ts $x \nprec y$, for all $x \in \Cr_1$ and $y \in \Cr_2$.
Equivalently, $\Lc$ is the linear extension of $P$ which assigns the smallest possible
values to the elements of $\Cr_2$.
Note that $\xsf_1$ in the previous recursion is equal to \ts $\Lc^{-1}(1)$.

\smallskip

Let  \ts $S:=(s_{i,j})_{i,j \geq 1}$ \ts and \ts $T:=(t_{i,j})_{i,j \geq 1}$ \ts
be the \ts $\mathbb{P} \times \mathbb{P}$ \ts matrices given by
\[  s_{i,j} \ := \
\begin{cases}
	\. 1 & \text{ if } i-j=1 \\
	\. 0 & \text{ if } i-j \neq 1
\end{cases} \.  \qquad \text{ and } \qquad
t_{i,j} \ := \
\begin{cases}
	\. 1 & \text{ if } i-j \leq  0 \\
	\. 0 & \text{ if } i-j > 0
\end{cases}\..
\]
In other words,
\begin{align*}
	S \ := \
	\begin{bmatrix}
		0 & 0 & 0 &   \\
		1 & 0 & 0 &   \ddots  \\
		0 & 1 & 0  &  \ddots \\
		  &  \ddots  &  \ddots & \ddots
	\end{bmatrix},
	\qquad
	T \ := \
	\begin{bmatrix}
		1 & 1 & 1 & \\
		0 & 1 & 1 &  \ddots\\
		0 & 0 & 1  & \ddots \\
		  &  \ddots  &  \ddots & \ddots
	\end{bmatrix}\..\end{align*}

\smallskip

Similarly, for $k\geq 1$, denote by  \ts $W_k:=(w_{i,j})_{i,j \geq 1}$ \ts
the  \ts $\mathbb{P} \times \mathbb{P}$ \ts matrix given by
\[ w_{i,j} \ := \
\begin{cases}
	1 & \text{ if $i=j \leq k$},\\
	0 & \text{ otherwise.}
\end{cases}
\]

In other words,
\begin{align*}
	W_1 \ := \
	\begin{bmatrix}
		1 & 0 & 0 &   \\
		0 & 0 & 0 &   \ddots  \\
		0 & 0 & 0  &  \ddots \\
		&  \ddots  &  \ddots & \ddots
	\end{bmatrix},
	\qquad
	W_2 \ := \
	\begin{bmatrix}
		1 & 0 & 0 & \\
		0 & 1 & 0 &  \ddots\\
		0 & 0 & 0  & \ddots \\
		&  \ddots  &  \ddots & \ddots
	\end{bmatrix}, \qquad \text{etc.}\end{align*}

\smallskip

\begin{definition}
Let \ts $x:=\Lc^{-1}(i)$.  \ts The \defn{characteristic matrices} \ts $M_1, \ldots, M_{\n}$ \ts of the poset~$P$
are defined as:
	\begin{equation}\label{eq:Mi}
			 M_i \ := \
		\begin{cases}
			\. S & \text{ if } \ x \in  \Cr_1\.,\\
			\. W_{\inc(x)+1} \. T & \text{ if }  \ x \in \Cr_2 \ \text{ and } \ x \neq \beta_{\bnb}\.,\\
			\. W_{\inc(x)+1} & \text{ if } \ x = \beta_{\bnb}\..
		\end{cases}
	\end{equation}
Note that \. $M_1,\ldots M_{\n}$ \. are nonnegative, nonzero matrices.
Also note that the products of these infinite matrices are well defined (as every entry below the first subdiagonal are equal to 0).
Finally, note that the product \. $M_i \vb $ \. is well defined for every vector
\ts $\vb := \bigl(\vbr(1), \vbr(2),\ldots\bigr)$ \ts with bounded support.
\end{definition}

\smallskip

\subsection{Product formula}
We now turn to the main result of this section.

\smallskip

\begin{lemma}\label{l:characteristic matrix}
	For every poset $P$ of width two, we have:
	\[ \bN_{P} \ = \   M_{1} \,  M_{2} \ \cdots \ M_{\dndt}\.,  \]
	where $\dndt:=\Lc^{-1}(\beta_{\bnb})$.
\end{lemma}

\smallskip

\begin{proof}
	We prove the lemma by induction on the value of $\dndt$.
	Let the base case be when  $\dndt$ is equal to $1$.
	In this case, we have  $\beta_1=\beta_{\bnb}$, and $P$ has exactly $\inc(\beta_{\bnb})+1$ linear extensions, namely the linear extensions $L_i$ ($i \in \{1,\ldots, \inc(\beta_{\bnb})+1\}$)
	for which $\beta_{\bnb}$ is the $i$-th smallest element of the linear extension.
	It then follows that, for all $i,j \geq 1$,
	\[ N_P(i,j) \ = \
	\begin{cases}
		\. 1 & \text{ if } \ i=j \  \text{ and } \ i \leq \inc(\beta_{\bnb})+1,\\
		\. 0 & \text{ otherwise.}
		\end{cases}  \]
	This implies that \. $\bN_P=W_{\inc(\beta_{\bnb})+1}$\.,
	which proves the base case.
	
	Now let $\xsf_1$ be the special element in the recursion outlined above,
	and
	let $P'$ be the induced subposet on  \ts $X':= X- \{\xsf_1\}$\ts\..	
	Note that the characteristic matrices \. $M_1',\ldots, M_{\dndt-1}'$ \.  of $P'$ satisfy
	\[  M_i' \ = \  M_{i+1} \qquad \text{ for } i \in \{1,\ldots, \dndt-1\}\..  \]
	Also note that, by the induction assumption, the matrix $\bN_{P'}$ for $P'$ satisfies
	\[ \bN_{P'} \quad = \quad   M_1' \, \cdots \, M'_{\dndt-1}\..\]
	Thus it suffices to show that
	\begin{equation}\label{eq:characteristic induction}
		\bN_P  \quad = \quad M_1 \bN_{P'}\..
	\end{equation}
	
We split the proof of \eqref{eq:characteristic induction} into two cases.
For the first case, suppose that \. $\xsf_1=\alpha_1$\..
	It then follows from \eqref{eq:recursion NP 1} that 	
	\.  $\bN_P= S \bN_{P'}$\..
	On the other hand, we have  \. $M_1=S$ \. by definition.
	Combining these two observations   proves  \eqref{eq:characteristic induction}
in the first case.
	
For the second case, suppose that \. $\xsf_1=\beta_1$\..
	It then follows from \eqref{eq:recursion NP 2} that 	
	 \. $\bN_P=W_{\inc_P(x_1)+1}T\bN_{P'}$\..
	On the other hand, we have \. $M_1=W_{\inc_P(x_1)+1}T$ \. by definition.
	Combining these two observations   proves \eqref{eq:characteristic induction}
in the second case.  This proves the induction step.
\end{proof}

\bigskip

\section{Cross--product relations}\label{s:cross product relation}

In this section we define an additional algebraic structures called cross--product relations,
which will be useful in checking if every $2\times 2$ minor of the matrix $\bF_P$
as in~\eqref{eq:CPconjecture-gen} is nonpositive.

\medskip

\subsection{Admissible vectors}
Let \ts $\vb= (\vbr(1), \vbr(2),\ldots) \in \nn^\pp$ \ts be a sequence of
nonnegative integers.  We say that $\vb$ is an \defn{admissible vector},
if
\[
\vbr(i) >0  \ \text{ and } \  \vbr(k) >0 \ \quad \text{ implies } \ \quad \vbr(j) >0\ts,
\]
for all \ts $ i \leq j \leq k$.
%
%
The \defn{support} \ts $\supp(\vb)$ \ts is  the set \. $\{\. i\in\pp~:~\vbr(i) >0 \. \}$.
For a nonzero admissible vector~$\vb$,
we denote by $\up(\vb)$ the smallest integer in the support of $\vb$,
and by $\down(\vb)$ the largest integer in the support of $\vb$.


\smallskip

\begin{definition}[{\rm\defn{cross--product relation}}]
\label{def:CP-relation}
For all admissible vectors \ts $\vb$ \ts and \ts $\wb$, we write \. $\vb \. \precCP \. \wb$ \. if,
for  all \ts $1 \leq i \leq j$, we have:
	\begin{equation}\label{eqcp determinant}
	   \vbr(i) \wbr(j)  \ - \ \vbr(j) \wbr(i)  \quad = \quad \det
\begin{bmatrix}
	\vbr(i) & \wbr(i)\\
	\vbr(j) & \wbr(j)
\end{bmatrix}  \quad \geq \quad 0\..
	\end{equation}
\end{definition}

\smallskip
Note that  \. $\precCP$ \. is not  a transitive relation, since
 we have \. $\vb \. \precCP \. \zero$ \. and \. $\zero \. \precCP \. \wb$ \. for all admissible vectors $\vb,\wb$, while
\. $\vb \precCP \wb$ \. does not  always hold.
However, the relation $\precCP$ will be a transitive relation when restricted to nonzero admissible vectors,
as shown in the next lemma.


\smallskip

\begin{lemma}\label{l:CP is transitive}
	For all nonzero admissible vectors \ts $\vb$ \ts and \ts $\wb$, we have:
	\begin{equation}\label{eqcp  updown}
		\text{$\vb \precCP \wb$ \quad implies \quad $\up(\vb) \leq  \up(\wb)$ \ \  and  \ \ $\down(\vb) \leq   \down(\wb).$}
	\end{equation}
	Furthermore, for all nonzero admissible vectors \. $\vb,\wb,\ub$\.,
	\begin{equation}\label{eqcp is transitive}
			 \vb \. \precCP \. \wb \ \text{ and } \ \wb \. \precCP \. \ub \qquad \text{ implies } \qquad  \vb \. \precCP \. \ub\..
	\end{equation}
\end{lemma}

\smallskip

\begin{proof}
	We first prove \eqref{eqcp  updown}.
	Fix \. $j \in \supp(\vb)$ \. (note that $j$ exists since $\vb$ is a nonzero vector),
	and let $i$ be an integer strictly smaller than \. $\up(\vb)$.
	Note that $i \leq j$ and $\vbr(i)=0$ by definition.
	Then, we have
	\[ 0 \quad = \quad \vbr(i) \wbr(j)  \quad \geq_{\eqref{eqcp determinant}} \quad    \vbr(j) \wbr(i)\.. \]
	Since $\vbr(j) >0$, it then follows from the equation above that
	$\wbr(i)=0$.
	Since the choice of $i \in \up(\vb)$ is arbitrary,
	it follows that \. $\up(\vb) \leq  \up(\wb)$\..
	The proof that \. $\down(\vb) \leq  \down(\wb)$ \. follows from an analogous argument.
	This concludes the proof of \eqref{eqcp  updown}.
	
	We now prove \eqref{eqcp is transitive}.
	Let $i, j$ be positive integers satisfying $1 \leq i \leq j$.
	It suffices to show that
	\begin{equation}\label{eqcp vu}
		\vbr(i) \, \rub(j) \ - \  \vbr(j) \, \rub(i) \ \geq \ 0\..
	\end{equation}
	We will assume without loss of generality that \ts $\vbr(j)>0$\ts.
	Indeed, if \. $\vbr(j) =0$\.,
	then   \. $\vbr(j) \, \rub (i) =0$\., and \eqref{eqcp determinant} follows immediately.
	By an analogous  reasoning, we will also assume that \. $\rub(i) >0$\..
	
	Since  \. $\vbr(j), \rub (i) >0$ \. and \. $\down(\vb) \leq \down(\wb) \leq  \down(\ub)$ \. (by \eqref{eqcp  updown}),
	it then follows that
	$\wbr(j) >0$  and
	$\rub(j)>0$.
	We can then apply \eqref{eqcp determinant} consecutively to \. $\vb \precCP \wb$ \. and \. $\wb \precCP \ub$ \. to get
	\[ \frac{\vbr(i)}{\vbr(j)} \quad \geq \quad  \frac{\wbr(i)}{\wbr(j)}  \quad \geq \quad  \frac{\rub(i)}{\rub(j)}\..  \]
	 This proves \eqref{eqcp vu}, and the proof is complete.
\end{proof}

\smallskip

\subsection{Multiplication properties}
We now collect several properties of the matrices $S$ and $T$
in relations to the cross--product relation.
Let $U= I -W_1$, which differs from the identity matrix by $U(1,1)=0$.
We now have
\begin{align}\label{eqST--TS}
	TS \ = \
		\begin{bmatrix}
		1 & 1 & 1 & \\
		1 & 1 & 1 &  \ddots\\
		0 & 1 & 1  & \ddots \\
		&  \ddots  &  \ddots & \ddots
	\end{bmatrix}\., \qquad ST \ = \  UTS  \ = \
	\begin{bmatrix}
	0 & 0 & 0 & \cdots \\
	1 & 1 & 1 &  \cdots\\
	0 & 1 & 1  & \ddots \\
	&  \ddots  &  \ddots & \ddots
\end{bmatrix}\..\
\end{align}
We also have, for all $k \geq 0$,
\begin{equation}\label{eqSW--WS}
	S\. W_{k}  \ = \ W_{k+1}S \.,
\end{equation}
and combining \eqref{eqST--TS} and \eqref{eqSW--WS} gives us
\begin{equation}\label{eqSWT--WTS}
	S\. W_{k}\. T  \ = \  W_{k+1} \. U \. T \. S\..
\end{equation}

\smallskip

It can be directly verified from the definition that
all \ts $2 \times 2$ \ts minors of matrices \ts $S$, $T$,
$W_k$ \ts and \ts $U$, are nonnegative.

\smallskip

\begin{lemma}\label{l:CPtransferrable}
	Let $M$ be a matrix such that
	\[
    M \ \in  \ \{ \. S,T,U\} \ \cup \ \{\. W_k \,~:~k \geq 1\. \}.
    \]
	Then, for all admissible vectors $\vb,\wb$, we have \. $M\vb$ \. and \. $M \wb$ \. are also admissible vectors.
	Furthermore,
	\begin{equation}\label{eq:CPCauchy-Binet}
		\vb \ \precCP \ \wb \qquad \text{ implies } \qquad
		M \vb \ \precCP \ M\wb\..
	\end{equation}
\end{lemma}
\smallskip
\begin{proof}
	It is straightforward to check by a direct computation that \. $M \vb$ \. and \. $M \wb$ \. are  admissible vectors.
	Now note that, by the \emph{Cauchy--Binet formula}~\eqref{eq:CB-formula} in this case, we have:
	 \begin{align*}
	 	\det   \begin{bmatrix}
	 		[M\vb](i) & [M\wb](i)\\
	 		[M\vb](j) & [M\wb](j)
	 	\end{bmatrix} \quad = \quad  \sum_{1 \leq k<\ell<\infty} \ \det \begin{bmatrix}
	 		M(i,k) & M(i,\ell) \\
	 		M(j,k) & M(j,\ell)
	 	\end{bmatrix} \ \det \begin{bmatrix}
	 		\vbr(k) & \wbr(k)\\
	 		\vbr(\ell) & \wbr(\ell)
	 	\end{bmatrix},
	 \end{align*}
for all \ts $1\leq i \leq j$.  Since \ts $\vb \. \precCP \. \wb$ \ts and every \ts $2 \times 2$ \ts
minor of $M$ is nonnegative, it follows that the right side of the equation above is nonnegative.
	 This completes the proof.
\end{proof}

\smallskip

\begin{lemma}\label{lemmonotone basic}
	For every admissible vector $\vb$ and every $k\geq 1$,
	we have \. $W_k\vb \ \precCP \ W_{k+1}U \vb $\ts.
\end{lemma}

\smallskip

\begin{proof}
	This follows from a straightforward computation.
\end{proof}

\smallskip

\begin{lemma}\label{l:Milogconcave}
	For every  admissible vector \ts $\vb\in \nn^\pp$ \ts and every \ts $1\le i \le n$, we have:
	\begin{equation*}
		M_iS \vb \ \precCP \  SM_i \vb\ts,
	\end{equation*}
	where $M_i$ is the characteristic matrix defined in \eqref{eq:Mi}.
\end{lemma}

\smallskip

\begin{proof}
	We split the proof into three cases.
	For the first case, suppose that \ts $M_i=S$.
	Then
	\. $M_iS=S^2=SM_i\.,$
	and the lemma immediately follows.
	For  the second case, suppose that
	\. $M_i= W_{k}T$ \. for some \. $k\geq 1$.
	We then have
	\[ M_i S \vb \  = \   W_kTS \vb \  \precCP_{\text{Lem}~\ref{lemmonotone basic}}  \  W_{k+1}UTS \vb \  =_{\eqref{eqSWT--WTS}} \  S W_k T\vb \  =  \  SM_i \vb. \]
	For the third case, suppose that $M_i=W_k$ for some $k \geq 1$.
	We then have
		\[ M_i S \vb \  = \   W_kS \vb \  =_{\eqref{eqSW--WS}}  \  SW_{k-1} \vb \quad \precCP \  S W_k \vb \  =  \quad SM_i \vb\ts, \]
		where the inequality \. $ SW_{k-1} \vb \, \precCP \, S W_k \vb $ \. follows from a direct computation.
	This completes the proof.
\end{proof}

\smallskip

\begin{lemma}
	For every nonzero  admissible vector $\vb$, we have:
	\begin{equation}\label{eq:monotone S}
		\bN_PS \ts \vb \  \precCP \   S \ts\bN_P \ts \vb\ts.
	\end{equation}
\end{lemma}

\smallskip

\begin{proof}
	For  all \. $i \in \{0,\ldots, \dndt\}$\.,
	we denote by $\vb_i$ the vector given by
	\begin{align}\label{eqvb definition}
	 \vb_i \quad & := \quad M_1 \,\cdots\, M_{i} \ S \ M_{i+1} \,\cdots\, M_{\dndt} \. \vb\..	\end{align}
	Note that each $\vb_i$ is an admissible vector by Lemma~\ref{l:CPtransferrable}.
	It  suffices to show that \. $\vb_{\dndt} \. \precCP \. \vb_0$\..
	
	Now note that, if either \. $M_{\dndt}\ts \vb$ \. or \. $M_{\dndt} \ts S\ts \vb$ \. is equal to the zero vector,
	then either \. $\vb_{0}$ \. or \. $\vb_{\dndt}$ \.
	is equal to the zero vector, and the lemma follows immediately.
	We now assume that
	\. $M_{\dndt}\ts\vb$ \. and \. $M_{\dndt} \ts S\ts\vb$ \. are nonzero vectors.
	Since all matrices $M_i$, for $i< \dndt$, and matrix~$S$ map nonzero admissible
    vectors to nonzero admissible vectors,
	 it then follows from~\eqref{eqvb definition} that $\vb_0,\ldots, \vb_{\dndt}$ are nonzero admissible vectors.
	Now note that, for all $ \in \{0,\ldots, \dndt-1\}$, we have:
	\begin{alignat*}{2}
		\vb_{i+1} \quad  &= \ && M_1 \,\cdots\, M_{i}  \ M_{i+1} \ S \ M_{i+2} \,\cdots\, M_{\dndt} \vb \\
		&  \precCP_{\text{Lem}~\ref{l:CPtransferrable}, \ \text{Lem}~\ref{l:Milogconcave}}
\quad && M_1 \,\cdots\, M_{i} \ S  \ M_{i+1}  \ M_{i+2} \ \ldots M_{\dndt} \.   \vb \  = \  \vb_{i}\..
	\end{alignat*}
	By Lemma~\ref{l:CP is transitive}, this implies that \. $\vb_{\dndt} \. \precCP \. \vb_0$\.,
	which completes the proof, as desired.	
\end{proof}

\bigskip

\section{Log-concavity}\label{s:log}

In this section we collect various variations of poset log-concave inequalities
that will be used in the first proof of Theorem~\ref{t:main}.

\medskip

\subsection{Stanley type inequalities}\label{s:log-Stanley}
%
%
Fix \ts $1 \leq k \leq \ell \leq \bnb$.  For every \ts $1\le t \le n $,
let \. $\rbb_{t} = (\rbbt_1,\rbbt_2,\ldots)$ \. be the vector  given by
\begin{equation}\label{eq:definition r}
	\rbbt_{t}(i) \ := \ \bigl|\{\.L\in \Ec(P)~:~L(\beta_k)=i \ \  \text{and} \ \ L(\beta_\ell)=t\.\}\bigr|\ts.
\end{equation}

\medskip

\begin{lemma}\label{l:rvector is decreasing}
    In notation above, $\rbb_{t}$ \. is an admissible vector. Furthermore, we have:
	\begin{align} \label{eq:logconcave r1}
        \rbb_t \, &\precCP \,  S\rbb_{m}
        \quad \text{for all} \ \ t-1 \leq m \ \ \text{and} \ \ 1\le m \le n\ts.
	\end{align}
\end{lemma}

\medskip

For every $x \in X$, denote by \. $\qb = \qb_{x} := \bigl(\qbr(1),\qbr(2),\ldots)$ \. the vector given by
\begin{equation}\label{eq:definition q}
	\qbr(i) \, = \, \qbr_x(i) \  := \  \bigl| \{ \. L\in \Ec(P)~:~L(x) =  i\}\bigr|\ts.
\end{equation}

\medskip

\begin{lemma}\label{l:logconcavity q}
	In notation above, $\qb_x$ is an admissible vector that satisfies
	\begin{equation*}
		\qb_x \quad \precCP \quad S \qb_x.
	\end{equation*}
	
\end{lemma}

\medskip

\begin{cor}[{\rm \defn{Stanley inequality}, \cite[Thm~3.1]{Sta}}] \label{c:logconcavity-Stanley}
For every poset $P=(X,\prec)$ of width two and element $x\in X$, we have:
\begin{equation}\label{eq:stanley}
\qbr_x(i)^2  \, - \,  \qbr_x(i-1) \. \qbr_x(i+1) \ \ge \ 0\ts.
\end{equation}
\end{cor}

\smallskip

\begin{proof}
In notation above, by Lemma~\ref{l:logconcavity q}, we have:
	\begin{align}\label{eqlogconcave 2}
		\qbr_x(i)^2  \ - \  \qbr_x(i-1) \. \qbr_x(i+1) \quad = \quad  \det \begin{bmatrix}
			\qbr_x(i) & [S\qb](i) \\
			\qbr_x(i+1) & [S\qb](i+1) \\
		\end{bmatrix}
		\quad  \geq  \quad 0,
	\end{align}
for all $i \geq 1$.
\end{proof}

\smallskip

\begin{rem}\label{r:log-concavity-Stanley} {\rm
Lemma~\ref{c:logconcavity-Stanley} is a special case of Stanley's original log-concavity
for general posets.  Stanley's proof uses the (non-elementary) \emph{Alexandrov--Fenchel inequality}
for mixed volumes, the approach was generalized in~\cite{KS} to prove inequality~\eqref{eq:KS-ineq}.
Thus  our approach  provides the first elementary proof of \eqref{eqlogconcave 2}
for width two posets (cf.~$\S$\ref{ss:finrem-CPC}).}
\end{rem}

\smallskip

\subsection{Setting up the argument}
 We now build toward the proof of Lemma~\ref{l:rvector is decreasing} and Lemma~\ref{l:logconcavity q}.

\smallskip

Let \. $\ab = \ab_P \ts = \ts \bigl(\abr(1),\abr(2),\ldots\bigr)$ \.  and \. $\bb:= \bb_P\ts = \ts \bigl(\bbr(1),\bbr(2),\ldots\bigr)$ \. be the vectors given by
\begin{align*}
	\abr(i) \ & := \   \bigl|\{\. L\in \Ec(P)~:~L(\beta_1) =  i\.\}\bigr|\.,\\
    \bbr(i) \ & := \   \bigl|\{\. L\in \Ec(P)~:~L(\beta_{\bnb}) =  i+ \lessr_P(\beta_{\bnb})\.\}\bigr|\..
\end{align*}

\smallskip

\begin{lemma}\label{l:logconcavity}
	In notation above, \. $\ab$ \ts and \ts $\bb$ \. are admissible vectors that satisfy
	\begin{equation*}
	  \ab \ \precCP \ S \ab \quad \text{ and } \quad \bb \ \precCP \ S \bb.
	\end{equation*}
\end{lemma}

\smallskip

\begin{proof}
	Let \ts $\one=(1,1,\ldots)$ \ts and observe that \ts $\one$ \ts is an admissible vector.
	Note that \. $\ab= \bN_P\one$ \. by definition.  Thus, Lemma~\ref{l:CPtransferrable}
    implies that $\ab$ is also admissible vector.
	
	For the first inequality, since  \. $\one \ts \precCP \ts  S\one$ \. from direct computation,
	it then follows that
		\begin{equation*}
		\begin{split}
		\ab \ = \ 	\bN_P\one \quad  \precCP_{\eqref{eq:CPCauchy-Binet}} \quad \bN_P S \one  \quad
			 \precCP_{\eqref{eq:monotone S}}
			\  S\bN_P \one \ = \ S\ab.
		\end{split}
	\end{equation*}
Since every vector in the equation above is nonzero,
it then follows from Lemma~\ref{l:CP is transitive}
that \. $\ab \precCP S\ab$\., as desired.

For the second inequality, let \ts $P':=(X,\precc_{P'})$ \ts be the order dual of~$P$,
i.e., \. $x \precc_{P'} y$ \.  if and only if  \. $y \prec_P x$ \. for all \ts $x, y \in X$.
Let $\ab' := \ab_{P'}$.
It follows from the duality that for all \ts $i \geq 1$, we have:
\begin{align*}
 \bbr(i) \  = \ \big|\{\. L \in \Ec(P)~:~L(\beta_{\bnb})=i+c+1\. \} \big| \ = \ \big|\{\. L' \in \Ec(P')~:~L'(\beta_{\bnb})=n-i-c\. \} \big|
 \ = \  \ab'(\n-i-c),
\end{align*}
where \. $c:=\lessr_P(\beta_{\bnb})-1$\..
Since $\ab'$ is an admissible vector, it then follows that $\bb$ is also an admissible vector.
Now note that, for all $1 \leq i \leq j$, we have:
\begin{align*}
&\det   \begin{bmatrix}
	\bbr(i) & [S\bb](i)\\
	\bbr(j) & [S\bb](j)
\end{bmatrix}  \   = \ 	 \det   \begin{bmatrix}
		\bbr(i) & \bbr(i-1)\\
		\bbr(j) & \bbr(j-1)
	\end{bmatrix} \ = \
		\det   \begin{bmatrix}
		\abr'(\n-i-c) & \abr'(\n-i-c+1)\\
		\abr'(\n-j-c) & \abr'(\n-j-c+1)
	\end{bmatrix} \\
& \quad = \  \det   \begin{bmatrix}
	[S\ab'](\n-i-c+1) & \abr'(\n-i-c+1)\\
	[S\ab'](\n-j-c+1) & \abr'(\n-j-c+1)
\end{bmatrix}
\  = \   \det   \begin{bmatrix}
	\abr'(\n-j-c+1)  & [S\ab'](\n-j-c+1)  \\
\abr'(\n-i-c+1)	& [S\ab'](\n-i-c+1)
\end{bmatrix}.
\end{align*}
Note that the rows of the matrix in the right hand side are in the increasing order.
On the other hand, we also have  \. $\ab'\. \precCP \. S\ab'$ \. from the first part of the lemma,
which implies that the right side of the equation above is nonnegative.
Thus we conclude that
 that \. $\bb \. \precCP \. S\bb$, as desired.
\end{proof}

\smallskip

\smallskip

Let \. $\bA_P$ \. and \. $\bB_P$ \. be the diagonal \ts $\pp \times \pp$ \ts matrices  given by
\[ \bAr_P(i,j) \ := \
\begin{cases}
	\abr(i) & \text{ if } i=j\\
	0 &  \text{ otherwise}
\end{cases} \qquad \text{and} \qquad
    \bBr_P(i,j) \ := \
	\begin{cases}
		\bbr(j) & \text{ if } i=j\\
		0 &  \text{ otherwise}
	\end{cases}    \]

\smallskip

\begin{lemma}\label{lemlogconcavity 2}
	 For every admissible vector $\vb$, we have:
	 \begin{equation}\label{eq:monotone ab}
	  \bA_P S\vb \, \precCP \, S\bA_P\vb  \qquad \text{ and } \qquad  \bB_P S\vb \, \precCP \, S\bB_P\vb.
	 \end{equation}
\end{lemma}

\smallskip

\begin{proof}
	We will show only the proof of the first inequality as the other inequality is analogous.
    For all $1 \leq i \leq j$, we have:
	\begin{align*}
	 & [\bA_PS\vb](i) \, \cdot \,  [S\bA_P\vb](j) \ = \
    \abr(i) \. \abr(j-1) \. \vbr(i-1) \. \vbr(j-1) \\
	 & \hskip1.cm \geq_{\text{Lem~\ref{l:logconcavity}}}
	 \  \abr(i-1) \. \abr(j) \. \vbr(i-1) \. \vbr(j-1)
	 \  = \  [S\bA_P\vb](i) \, \cdot \, [\bA_P S\vb](j).
	\end{align*}
	This proves the claim.
\end{proof}

\medskip


\subsection{Proof of Lemma~\ref{l:rvector is decreasing}}
Let \. $1 \le i \le t \le n$, and let $L\in \Ec(P)$ be a fixed linear extension of~$P$,
such that \. $L(\beta_k)=i$ \.  and \. $L(\beta_\ell)=t$.
		We will decompose $L$ into three linear extensions $L_1, L_2, L_3$ (of smaller posets),
		where the linear extension $L_1$ will encode the total ordering of elements before $L(\beta_k)$,
		the linear extension $L_2$ will encode the total ordering of elements between $L(\beta_k)$ and $L(\beta_\ell)$,
		and the linear extension $L_3$ will encode the ordering of elements after $L(\beta_\ell)$.
		
	Let $P_1, P_2, P_3$ be the induced subposet of $P$ on the subsets of $X$ given by
	\begin{align*}
		X_1 \ &:=   \  X \.\sm  \.  \{ \. x \in X \.~:~\.  x \ \succc_P \  \beta_k   \. \},\\
		X_2 \ &:=   \  X \.\sm  \. \{ \.  x \in X \.~:~\.  \. x \ \prec_P \ \beta_k  \ \, \text{ or } \ \, x \ \succ_P \ \beta_\ell \. \},\\
		X_3 \ &:=   \  X \.\sm  \.  \{ \. x \in X \.~:~\.  x \ \precc_P \ \beta_\ell  \. \}.
	\end{align*}
	Note that $\beta_k$ is a maximal element of $X_1$,
	that $\beta_k$ is a minimal element of $X_2$ and $\beta_\ell$ is a maximal element of $X_2$,
	and that $\beta_\ell$ is a minimal element of $X_3$.
	Note also that $X_1$  contains all elements of $X$ that are smaller than $\beta_k$ w.r.t.\
    the linear extension $L$, that $X_2$ contains all elements all of $X$ that lie between $\beta_k$
    and $\beta_\ell$ w.r.t.~$L$, and that $X_3$ contains all elements of $X$ that are greater
    than  $\beta_\ell$ w.r.t.~$L$.

    Let $P_1$, $P_2$ and $P_3$ be the restrictions of $P$ to $X_1$, $X_2$ and $X_3$, respectively.
	Similarly, let \ts $L_1\in \Ec(P_1)$, \ts $L_2\in \Ec(P_2)$ and $L_3\in \Ec(P_3)$ \ts
    be the restrictions of $L$ to \ts $X_1$, $X_2$ and $X_3$, respectively.
	Note that the three linear extensions satisfy the following equations:
	\begin{equation}\label{eqtriplet Ls}
	\begin{split}
		& L_1(\beta_k) \, = \, i \,,  \quad L_2(\beta_k)  \, = \,  i- \lessr_P(\beta_k)\,, \quad  L_2(\beta_\ell) \, = \,  t-\lessr_P(\beta_k)\,, \quad L_3(\beta_\ell) \, = \, t-\lessr_P(\beta_\ell)\.,
	\end{split}
	\end{equation}
	because in $X_2$  elements strictly less than $\beta_k$ are removed, and in $X_3$ elements strictly less than $\beta_\ell$ are removed.
	On the other hand, given a triplet $(L_1,L_2,L_3)$ that satisfies \eqref{eqtriplet Ls}, we can recover the original linear extension $L$ by
	\begin{align*}
		L(x) \quad = \quad
		\begin{cases}
			L_1(x) & \text{ if } \ L_1(x) \leq i,\\
			L_2(x) + \lessr_P(\beta_k) & \text{ if } \ i -\lessr_P(\beta_k) \. \leq \. L_2(x)  \. \leq \. t- \lessr_P(\beta_k),\\
			L_3(x) + \lessr_P(\beta_\ell) & \text{ if } \ L_3(x)  \. \geq \. t- \lessr_P(\beta_\ell).
		\end{cases}
	\end{align*}
	It follows from \eqref{eqtriplet Ls} that $L$ is well-defined and is a linear extension of $X$.
	This shows that the given correspondence associating $L$ to $(L_1, L_2, L_3)$ is a bijection.	
	It then follows from the correspondence above that for all \ts $i \geq 1$,
	\begin{equation}\label{eqexpansion ru}
	\begin{split}
		\rbbt_t(i) \ & = \ \big|\{\. L \in \Ec(P) \.~:~\.    L(\beta_k)=i, \ L(\beta_\ell)=t\. \}\big| \\	
		 & = \ \big|\{\. L_1 \in \Ec(P_1) \.~:~\.    L_1(\beta_k)=i\. \} \big| \, \cdot \,  \big|\{\. L_2 \in \Ec(P_2) \.~:~\.     L_2(\beta_k)=i-c_1, \ L_2(\beta_\ell)=t-c_1\. \} \big|	 \\
		 & \hskip5.27cm  \cdot \, \big|\{\. L_3 \in \Ec(P_3) \.~:~\.    L_3(\beta_\ell)=t-c_2\. \} \big| \\
		&= \ \bbr_{P_1}(i) \
		\bN_{P_2}(i-c_1, t-c_2)  \ \abr_{P_3}(t-c_2)\.,
	\end{split}
	\end{equation}
	where \. $c_1:=\lessr_P(\beta_k)$ \. and \. $c_2:=\lessr_P(\beta_\ell)=c_1+\lessr_{P_2}(\beta_\ell)$.

	Let  $\eb_1, \eb_2,\ldots$ be the standard unit vectors for $\Rb^\pp$, and let $\vb$ and $\wb$ be two admissible vectors  given by
	\[
        \vb \ := \  \abr_{P_3}(t-1-c_2) \.\eb_{t-1-c_2} \quad \text{and} \quad  \wb \ := \  \abr_{P_3}(m-c_2) \.\eb_{m-c_2}\ts.
    \]
	Note that \. $ \vb \ts \precCP \ts \wb$ \. by the assumption that $t-1 \leq m$.
	Also note that, from \eqref{eqexpansion ru}, we have:
	\begin{equation*}
			 \rbb_t \ = \  \bB_{P_1}  S^{c_1}  \bN_{P_2} \. S \. \vb
            \quad \text{and} \quad S\rbb_{m} \ = \  S \. \bB_{P_1}   S^{c_1}  \bN_{P_2}  \wb.
	\end{equation*}
	It then follows from the equation above that $\rbb_t$ is an admissible vector.
	
	If either  \. $\bN_{P_2} S \vb$ \. or \. $\bN_{P_2}\wb$ \. is equal to the zero vector,
	then either \ts $\rbb_t$ \ts or \ts $\rbb_m$  is equal to the zero vector,
	and the lemma follows immediately.
	We now assume that   \. $\bN_{P_2} S \vb$ \. and \. $\bN_{P_2}\wb$ \ts are nonzero vectors.
	Then we have:
	\begin{align*}
	\begin{split}
		  \rbb_t \ = \  \bB_{P_1}  S^{c_1} \bN_{P_2} S \vb \
		& \precCP_{\eqref{eq:monotone S}} \   \bB_{P_1} S^{c_1+1} \bN_{P_2}\vb  \
        \precCP_{\eqref{eq:monotone ab}} \  S \.\bB_{P_1} S^{c_1} \bN_{P_2} \vb \\
		& \precCP_{\eqref{eq:CPCauchy-Binet}}    \   S \. \bB_{P_1}  S^{c_1} \bN_{P_2} \wb \ = \ S \rbb_m.
	\end{split}
	\end{align*}
	Note that  every vector in the equation above is nonnegative.
	It then follows from Lemma~\ref{l:CP is transitive} that
	\. $\rbb_t \precCP \rbb_m$, as desired. \qed

\medskip

\subsection{Proof of Lemma~\ref{l:logconcavity q}}
By exchanging the label of $\Cr_1$ and $\Cr_2$ if necessary,
we can assume without loss of generality that $x\in\Cr_2$.
Let $k$ be the integer such that $x=\beta_k$.
By adding an extra maximum element to the poset if necessary, we can assume that $k< \bnb$.
Let $\ell := \bnb$.
	
Let $P_1, P_2, P_3$, and $c_1$ be as in the proof of Lemma~\ref{l:rvector is decreasing}.
It then follows from  the argument analogous to the proof of Lemma~\ref{l:rvector is decreasing},
that
\begin{equation}\label{eq:echo 1}
		\qb_x \ = \ \bB_{P_1}  S^{c_1} \bN_{P_2} \vb,
\end{equation}
where \ts $\vb=\ab_{P_3}$\ts.
	Since $\vb=\ab_{P_3}$ is an admissible vector by Lemma~\ref{l:logconcavity},
	it then follows from \eqref{eq:echo 1} and Lemma~\ref{l:CPtransferrable} that
	$\qb_x$ is an admissible vector.
	
We can always assume that\ts $\vb$ \ts is a nonzero vector.  Indeed, if \ts
$\vb$ is a zero vector, then both \ts $\qb$ \ts and \ts $S\qb$ \ts are equal
to the zero vector, and the lemma follows immediately.

Now note that
	\begin{equation}\label{eq:echo 2}
	\begin{split}
	\qb \quad  = \quad 	 \bB_{P_1}   S^{c_1} \bN_{P_2} \.\vb  \ & \precCP \ \bB_{P_1}  S^{c_1}  \bN_{P_2} S\. \vb  \
\precCP_{\eqref{eq:monotone S}} \ \bB_{P_1} S^{c_1+1} \bN_{P_2} \. \vb  \\
		  &  \precCP_{\eqref{eq:monotone ab}} \quad  \ S \.\bB_{P_1} S^{c_1} \bN_{P_2}  \vb
		 \  = \   S \qb.
	\end{split}
	\end{equation}
	Also note that  every vector in the equation above are nonzero vectors by assumption.	
	It then follows from Lemma~\ref{l:CP is transitive} that \. $\qb \. \precCP  \. S\qb$, as desired. \qed

\bigskip

\section{Algebraic proof  of Theorem~\ref{t:main-gen}}
\label{s:proof of CPC}

\subsection{Matrix formulation}
Let \ts $z_1,z_2,z_3\in X$ \ts be the fixed elements in the Cross--product Conjecture~\ref{conj:CP},
and let $|X|=n$.
Consider two \ts $\pp \times n$ \ts matrices \ts $\bG=\bG_P$ \ts and \ts $\bH= \bH_P$, with entries
\begin{equation}\label{eq:definition G}
\begin{split}
 G_P(i,t) \ := \ \bigl| \{ \. L\in \Ec(P)~:~L(z_1)=t-i, \, L(z_2)=t \. \}\bigr|, \\
 H_P(i,t) \ := \ \bigr|\{ \. L\in \Ec(P)~:~L(z_3)=t+i, \, L(z_2)=t \. \}\bigr|.
\end{split}
\end{equation}
These matrices are related to the matrix $\bF_P$ in Theorem~\ref{t:main} in the following way.
Let $i, j \geq 1$, $1\le t \le n$, and let \ts $L\in \Ec(P)$ \ts such that $L(z_1)=t-i$,
$L(z_2)=t$, and $L(z_3)=t+j$.  Note that $L\in \Fc(i,j)$.
We will split $L$ into two linear extensions of smaller posets,
with the former encoding the total ordering for elements before $L(z_2)$,
and the latter encoding the total ordering elements after $L(z_2)$.

Let $Q$ and $R$ be the induced subposets of $P$ on the sets
\begin{align*}
	X_1 \ := \ X - \{\. x \in X \. \mid \. x \succc_P z_2 \. \} \, \quad \text{and} \quad X_2  \ := \  X - \{\. x \in X \. \mid \. x \. \precc_P  \. z_2  \. \}\..
\end{align*}
Let  $L_1$ and $L_2$ be the restrictions of $L$ onto subsets $X_1$ and~$X_2$,
respectively.  We write \ts $(L_1,L_2)=\eta(L)$.  Note that $L_1$ and $L_2$ satisfy
\begin{alignat*}{2}
	L_1(z_2) \ &= \ t, \qquad &&L_1(z_1) \ = \  t-i,\\
	L_2(z_2) \ &= \ t-\lessr_P(z_2), \qquad && L_2(z_3) \ = \  t-\lessr_P(z_2)+j.
\end{alignat*}
On the other hand, given a pair \ts $(L_1, L_2)$ \ts that satisfies the
equation above, we can recover the original linear extension \ts $L$ \ts as
\[ L(x) \ = \
\begin{cases}
	L_1(x) & \text{ if } \  x \in X_1 \text{ and } L_1(x) \leq t,\\
	L_2(x)+\lessr_P(z_2) & \text{ otherwise.}
\end{cases}   \]
Hence the correspondence \ts $\eta: L \to (L_1,L_2)$ \ts as above is a bijection.

\smallskip

Let  \. $c:=\lessr_P(z_2)$\..
It then follows from the correspondence~$\eta$ above, that
\begin{align*}
	\aF_P(i,j)  \ & =   \  \sum_{t=1}^{\n} \ \bigl|\{\. L \in \Ec(P) \,~:~\, L(z_1) =  t-i,  \, L(z_2) =  t, \, L(z_3) =  t+j \.  \}\bigr| \\
	& = \ \sum_{t=1}^{\n} \ \bigl|\{\. L_1 \in \Ec(Q) \,~:~\,  L_1(z_1) =  t-i, \, L_1(z_2) = t \.    \}\bigr| \ \,\times \\
& \hskip2.5cm \times \
\bigl|\{\. L_2 \in \Ec(R) \,~:~\,  L_2(z_2) =  t-c, \, L_2(z_3) = t-c+j \.    \}\bigr| \\
	&=  \ \sum_{t=1}^{\n} \   G_{Q}(i,t) \,\. H_{R}(j,t-c)\.,
\end{align*}
for all \ts $i,j \geq 1$.  This is equivalent to
\begin{equation}\label{eq:formula FGH}
 \bF_P \ = \ \bigl(\bG_{Q}\bigr) \.\ts S^c \.\ts \bigl(\bH_{R}\bigr)^\top.
\end{equation}
Use the definition of \ts $S$ \ts to expand~\eqref{eq:formula FGH} as a sum,
and then apply Cauchy--Binet formula~\eqref{eq:CB-formula} to it.
We conclude:
{\small
\begin{equation}\label{eq:formula FGH Cauchy-Binet}
\begin{split}
	 & \det \begin{bmatrix}
		\aF_P(i,j) & \aF_P(i,\ell)\\
		\aF_P(k,j) & \aF_P(k,\ell)
	\end{bmatrix} \ = \  \sum_{1 \leq t \leq m \leq n} \.
\det \begin{bmatrix}
	\aG_{Q}(i,t) & \aG_{Q}(i,m)\\
	\ts\aG_{Q}(k,t) & \ts\aG_{Q}(k,m)
\end{bmatrix}  \ \det \begin{bmatrix}
	\aH_{R}(j,t-c) & \aH_{R}(\ell,t-c)\\
	\.\aH_{R}(j,m-c) & \.\aH_{R}(\ell,m-c)
\end{bmatrix}\ts,
\end{split}
\end{equation}}
for all \ts $1 \leq i \leq k$ \ts and \ts $1 \leq j \leq \ell$.
This reduces Theorem~\ref{t:main-gen} to checking signs of
\ts $2 \times 2$ \ts  minors of matrices \ts $\bG_{Q}$ \ts and \ts $\bH_{R}$
\ts separately.

\smallskip

\subsection{Matrix \ts $\bG_P$ \ts minors}
We now show that all \ts $2\times 2$ \ts minors of \ts $\bG_P$ \ts are
nonnegative, for all~$P$.  We start with the following lemma covering
a special case of this claim.

\smallskip

\begin{lemma}\label{l:crossproduct G}
Let $P=(X,\prec)$ be a poset of width two, and let $\bG_P$ be a matrix
defined in~\eqref{eq:definition G}.  We have:
\[		
\det \. \begin{bmatrix}
		G(i,t-1) & G(i,t)\\
		G(j,t-1) & G(j,t)
	\end{bmatrix}
	\ \geq \  0\ts,
\]
for all \ts $1\leq i \leq j$ \ts and \ts $1\le t \le n$.
\end{lemma}

\smallskip

\begin{proof}
Without loss of generality, assume that $z_1 \in \Cr_2$, since we can exchange the labels of
$\Cr_1$ and $\Cr_2$ otherwise.   We split the proof into two cases.
	
First, suppose that \ts $z_2 \in \Cr_2$.  Let $k,\ell$ be integers such that
	\[
    z_1 \, = \, \beta_k \qquad \text{ and } \qquad z_2 \, = \, \beta_\ell\,.
    \]
Let \. $\rbb_{t} :=\rbb_{k,\ell,t}$ \. be the vector defined in~\eqref{eq:definition r}.
It then follows from the definition that \. $G(i,t)= \rbbt_t(t-i)$, for all \ts $i \geq 1$ \ts
and \ts $1\le t\le n$.  It then follows that the given minor of $\bG_P$ is equal to
\begin{align*}
		\det \.   \begin{bmatrix}
			G(i,t-1) & G(i,t)\\
			G(j,t-1) & G(j,t)
                \end{bmatrix}  \ = \
		\det \.   \begin{bmatrix}
			\rbbt_{t-1}(t-1-i) & \rbbt_{t}(t-i)\\
			\rbbt_{t-1}(t-1-j) & \rbbt_t(t-j)
		          \end{bmatrix} \ = \
		\det \. \begin{bmatrix}
			[S\rbb_{t-1}](t-i) &  \rbbt_{t}(t-i)\\
			[S\rbb_{t-1}](t-j) & \rbbt_{t}(t-j)
		\end{bmatrix}.
	\end{align*}
Note that the rows of the matrix in the right hand side are in the decreasing order.
	On the other hand, we also have
	\. $\rbb_{t} \. \precCP \. S\rbb_{t-1}$ \. from~\eqref{eq:logconcave r1}.
	Combining these two observations,
	we conclude that the determinant above is nonnegative, as desired.

Second, suppose that \ts $z_2 \in \Cr_1$.
Let \ts $1\le k \le \bnb$ \ts and \ts $1\le h \le \ana$ \ts be such that
\[ z_1 \, = \, \beta_k \qquad \text{ and } \qquad z_2 \, = \, \alpha_h\,.
\]
 	Since the determinant in the lemma involves counting only linear extensions that satisfy \. $L(z_2) \in \{t-1,t\}$,
 	without loss of generality we can assume that
 	 \. $t-1 \. \leq \. L(z_2) \. \leq \.  t$\..
 	 This is equivalent to assuming that
$$
(\ast) \qquad
\beta_{t-h-1} \ \prec \ z_2 \  \prec \  \beta_{t-h+1}\..
$$
Let \. $\ell  := t-h$\..
Under the assumption $(\ast)$ above, it then follows that
 	\begin{alignat*}{3}
 		&  L(z_2)= t \quad  && \text{ is equivalent to } \quad   && L(\beta_\ell) < t, \quad \text{ and } \\
 		&  L(z_2)= t-1 \quad  && \text{ is equivalent to } \quad   && L(\beta_\ell)\geq t.
 	\end{alignat*}
 	Let \. $\rbb_{u} :=\rbb_{k,\ell,u}$ \. be the vector defined in \eqref{eq:definition r}, for $u \geq 1$.
 	It then follows that, under this scenario, for all $i\geq $1,
 	\begin{alignat*}{2}
 			G(i,t-1) \ &= \ \bigl| \{ \. L\in \Ec(P)~:~L(\beta_k)=t-1-i, \, L(\beta_\ell)<t \. \}\bigr| \ &&= \ \sum_{u =t}^{\infty} \, \rbbt_{u}(t-1-i)\ts,\\
 			 G(i,t) \ &= \ \bigl| \{ \. L\in \Ec(P)~:~L(\beta_k)=t-i, \, L(\beta_\ell)\geq t \. \}\bigr| \ &&= \ \sum_{v =1}^{t-1} \, \rbbt_{v}(t-i)\ts.
 	\end{alignat*} 	

 	The minor of $\bG_P$ as in the lemma is then equal to
 		\begin{align*}
 		& \det \.
 		\begin{bmatrix}
 			G(i,t-1) & G(i,t)\\
 			G(j,t-1) & G(j,t)
 		\end{bmatrix}
 		\ = \
 		\det \.
 		\begin{bmatrix}
 			\displaystyle \sum_{u =t}^{\infty} \rbbt_{u}(t-1-i)  & \displaystyle \sum_{v =1}^{t-1} \rbbt_{v}(t-i)\\
 			\displaystyle \sum_{u =t}^{\infty} \rbbt_{u}(t-1-j) & \displaystyle \sum_{v =1}^{t-1} \rbbt_{v}(t-j)
 		\end{bmatrix}\\
 		&\qquad = \
 		\sum_{u=t}^\infty\, \sum_{v=1}^{t-1} \. \det \. \begin{bmatrix}
 			\rbbt_u(t-1-i) &  \rbbt_{v}(t-i)\\
 			\rbbt_u(t-1-j) & \rbbt_{v}(t-j)
 		\end{bmatrix}
 		\ = \
 		\sum_{u=t}^\infty \, \sum_{v=1}^{t-1} \. \det \.
 		\begin{bmatrix}
 		[S\rbb_u](t-i) &  \rbbt_{v}(t-i)\\
 		[S\rbb_u](t-j) & \rbbt_{v}(t-j)
 		\end{bmatrix}.
 	\end{align*}
Again, note that the rows of the matrices in the right hand side is in the decreasing order.
 	On the other hand, we also have
 	\. $\rbbt_{v} \. \precCP \. S\rbbt_{u}$, for all \ts $v < u$ \ts from~\eqref{eq:logconcave r1}.
 	Combining these two observations,
 	we conclude that the determinant above is nonnegative, as desired.
 	This completes the proof of the second case.
\end{proof}

\smallskip

To generalize the lemma to all \ts $2\times 2$ \ts minors, we need the following technical result.

\smallskip

\begin{lemma}\label{l:support g}
Let \. $\gb_t := \bigl(\gbr_t(1),\gbr_t(2), \ldots \bigr)$ \. be
the vector  given by \. $\gbr_t(i) \. := \. G(i,t)$, for all \ts $i\ge 1$
\ts and \ts $1\le t\le n$.
Then \. $\gb_t$ \. is an admissible vector, for all \ts $1\le t\le n$.
Furthermore, the set
\[
\{ \.  t \in [n] \. ~:~\. \gb_t \text{ is a nonzero admissible vector} \. \}
\]
	is a closed interval of integers.
\end{lemma}

\smallskip

\begin{proof}
	Again, without loss of generality assume that $z_2 \in \Cr_2$.
	Let $\ell$ be the integer such that \. $z_2 = \beta_\ell$\..
	Since \ts $\gb_t$ \ts counts only  linear extensions satisfying \ts $L(z_2)=t$,
	without loss of generality we can assume that
 	\[
        \beta_{t-h} \ \prec \ z_2 \  \prec \  \beta_{t-h+1}\..
    \]
 	It then follows that \. $\gbr_t(i) = \qbr_{z_1}(t-i)$ \. for all \ts $i \geq 1$,
 where $\qb$ is defined in~\eqref{eq:definition q}.
Since \ts $\qb_{z_1}$ \ts is an admissible vector from  Lemma~\ref{l:logconcavity q},
 	it then follows that \ts $\gb_t$ \ts is also an admissible vector.
This proves the first part.
 	
For the second part, note that
 \[
 \qbr_{z_2}(t)   \ = \ \sum_{i \geq 1} \, \gbr_t(i) \quad \text{for all \ \. $1\le t\le n$}\ts.
 \]
 	Hence $\gb_t$ is a nonzero vector if and only if
 	$\qb_{z_2}(t)$ is nonzero.
 	On the other, we have that $\qb_{z_2}$ is
 	an admissible vector by  Lemma~\ref{l:logconcavity q}.
 	The second claim now follows by combining these two observations.
 \end{proof}

\smallskip

\begin{lemma}\label{l:minor G}
Every \ts $2 \times 2$ \ts minor of \ts $\bG_P$ \ts is nonnegative.
\end{lemma}

\smallskip

\begin{proof}

	Note that it suffices to show that
	\. $\gb_t \. \precCP \. \gb_m$ \. for all \ts $1 \leq t \leq m \leq \n$.
	The claim is vacuously true if either $\gb_t$ or $\gb_m$ is equal to zero,
	so we assume that both $\gb_t$ and $\gb_m$ are nonzero vectors.
	It then follows from Lemma~\ref{l:support g}
	that \ts $\gb_{t}$, $\gb_{t+1}$, \ldots, $\gb_{m}$ \ts are nonzero admissible vectors.
	On the other hand, we have\. $\gb_{i} \precCP \gb_{i+1}$ \. for all \ts $t \le i \le m-1$ by Lemma~\ref{l:crossproduct G}.
	It then follows from Lemma~\ref{l:CP is transitive} that
	\. $\gb_t \. \precCP \. \gb_m$. This implies the result.
\end{proof}

\smallskip

\subsection{Matrix \ts $\bH_P$ \ts minors}
This case follows via reduction to
the previous case.

\medskip

\begin{lemma}\label{l:minor H}
Every \ts $2 \times 2$ \ts minor of \ts $\bH_P$ \ts is nonpositive.
\end{lemma}

\smallskip

\begin{proof}
Let \ts $P^\ast:=(X,\precc^\ast)$ \ts be the \defn{order dual} of~$P$
obtained by reversing~$\prec_P$.
Let \ts $z_1^\ast \gets z_3$, \ts $z_2^\ast \gets z_2$ \ts and \ts $z_3^\ast \gets z_1$.
Similarly, let \ts $\bG^\ast = \bG_{P^\ast}$ \ts be the matrix in~\eqref{eq:definition G}
that corresponds to poset $P^\ast$ and elements $z_1^\ast, z_2^\ast, z_3^\ast$.
Therefore, \ts
$$\bH_P(i,t) = \  \bG(i,n-t+1) \quad \text{for all \ $i \geq 1$ \ and \ $1\le t \le n$. }
$$
Hence we have:
	\begin{align*}
		\det \.
		\begin{bmatrix}
			\aH(i,t)  &  \aH(i,m)\\
			\aH(j,t) & \aH(j,m)
		\end{bmatrix}
	\ \ =& \ \
			\det \.
	\begin{bmatrix}
		G^\ast(i,\n-t+1)  &  G^\ast(i,\n-m+1)\\
		G^\ast(j,\n-t+1) & G^\ast(j,\n-m+1)
	\end{bmatrix} \\
	 \ \ = & \ - \ \det \.
	\begin{bmatrix}
		G^\ast(i,\n-m+1)  &  G^\ast(i,\n-t+1)\\
		G^\ast(j,\n-m+1) & G^\ast(j,\n-t+1)
	\end{bmatrix},
	\end{align*}
for every $1\leq i \leq j$  and $1 \leq t \leq m\le n$.
In the second equality, we swap the first row and the second row
of the matrix, so that the rows and columns are indexed in the
increasing order. It then follows from Lemma~\ref{l:minor G}
that the determinant above is nonpositive, as desired.
\end{proof}

\smallskip

\subsection{Proof of Theorem~\ref{t:main-gen}}
	Let \ts $\bG_{Q}$ \ts and \ts $\bH_{R}$ \ts  be as in~\eqref{eq:formula FGH}.
	 Note that every \ts $2 \times 2$ \ts minor of \ts $\bG_{Q}$ is nonnegative
    by Lemma~\ref{l:minor G}, every \ts $2 \times 2$ \ts minor of \ts $S^c$ \ts
    is $0$ or~$1$, and every \ts $2 \times 2$ \ts minor of \ts $\bH_{R}$ is nonpositive
    by Lemma~\ref{l:minor H}.  By the Cauchy--Binet formula in~$\S$\ref{ss:back-CB},
    this implies that every \ts $2 \times 2$ \ts minor of \ts $\bF_P$ \ts
    is nonpositive, as desired.

    To make this argument even more explicit, the RHS of~\eqref{eq:formula FGH Cauchy-Binet}
    is a sum of products of nonnegative numbers with nonpositive numbers.
    This sum is thus a nonpositive number, which proves the result. \qed

\bigskip

\section{Lattice paths preliminaries}\label{s:lattice}

In this section we interpret the linear extensions of $P$ as monotonic lattice paths
and setup towards the proof of Theorem~\ref{t:main-big-q} given in the next section.

\smallskip

\subsection{Lattice path interpretation}\label{ss:lattice path interpretation}

Recall the notation for posets $P$ of width two given in~$\S$\ref{ss:back-width-2},
with two chains $\Cr_1$ and $\Cr_2$.  Denote by \ts $\zero=(0,0)$ \ts the origin and by
\ts $\eone = (1,0)$, \ts
$\etwo=(0,1)$ \ts two standard unit vectors in $\Zb^2$.

Informally, the lattice path is obtained from a linear extension $L$ by interpreting it as a sequence of North and East steps, where the step at position $k$ is North if and only if $L^{-1}(k) \in \Cr_2$.
Formally, let $L\in \Ec(P)$.  We associate to $L$ a \defn{North--East {\rm (NE)} lattice path} \.
$\phi(L):=(Z_t)_{1 \leq t \leq \n}$ \. in \ts $\Zb^2$ \ts
from \ts $\zero=(0,0)$ to $(\ana,\bnb)$.  The path \ts
$(Z_t)=\bigl(Z_t(1),Z_t(2)\bigr)$ \ts is
defined recursively as follows:
\begin{equation*}
	Z_0 \, = \, \zero, \qquad  Z_t \ := \
	\begin{cases}
		\ts Z_{t-1} \ts + \ts \eone & \text{ if } \ L^{-1}(t) \in \Cr_1\.,\\
		\ts Z_{t-1}\ts + \ts \etwo & \text{ if } \ L^{-1}(t) \in \Cr_2\..
	\end{cases}
\end{equation*}

We now characterize all the lattice paths that arise from this correspondence.

Denote by $\Cen(P)$ the set
\begin{align*}
 \Cenup(P) \ &:= \ \bigg\{ \left(h-\frac{1}{2}, k - \frac{1}{2}\right) \in \Rb^2 \ :  \  \alpha_h \. \precc_P \. \beta_k\,, \ 1\le h \le \ana, \ 1\le k \le \bnb \. \bigg \}\., \\
  \Cendown(P) \ &:= \ \bigg\{ \left(h-\frac{1}{2}, k - \frac{1}{2}\right) \in \Rb^2 \ :  \  \alpha_h \. \succc_P \. \beta_k\,, \ 1\le h \le \ana, \ 1\le k \le \bnb \. \bigg \}\..
\end{align*}
Let $\Forbup(P)$ and $\Forbdown(P)$ be  the set of unit squares in $[0, \ana] \times [0,\bnb]$
whose centers are in $\Cenup(P)$ and $\Cendown(P)$, respectively.
Note that the region $\Forbup(P)$ lies above the region $\Forbdown(P)$, and their interiors do not intersect.
Let $\Reg(P)$ be the (closed) region of $[0,\ana]\times [0,\bnb]$ that
is bounded from above by the region $\Forbup(P)$, and from below by the region $\Forbdown(P)$,
see Figure~\ref{f:regionP}.

It follows directly from the definition that $\Reg(P)$ is a connected row and
column convex region, with boundary defined by two lattice paths. Indeed,
the upper boundary is the lattice path corresponding to the minimal linear
extension $\Lc$ from $\S$\ref{ss:char-main-def}, and the lower boundary
is the lattice path corresponding to the minimal linear extension with
the labels of $\Cr_1$ and $\Cr_2$ exchanged.

\begin{figure}[hbt]
	\centering
	\begin{tabular}{c@{\hskip 1 in}c }
		\includegraphics[width=0.18\linewidth]{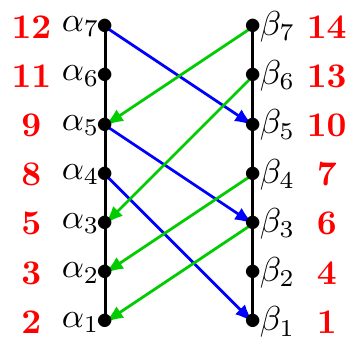}  &
		\includegraphics[width=0.2\linewidth]{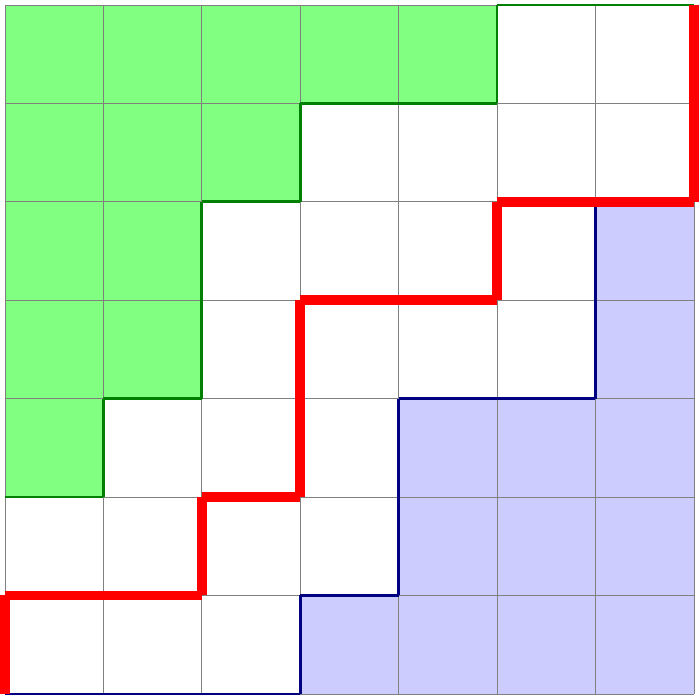}  \\
		(a) & (b)
	\end{tabular}
	\caption{(a) The Hasse diagram of a poset $P$, and a linear extension $L$ of $P$ (written in red).
	(b) The corresponding region $\Reg(P)$, with $\Forbup(P)$ in green and $\Forbdown(P)$ in blue, and the lattice path $\phi(L)$ in red.	
	}
	\label{f:regionP}
\end{figure}

\smallskip

\begin{lemma}\label{l:interpretation lattice path}
	The map~$\phi$ described above is a bijection between
    \ts $\Ec(P)$ \ts and NE lattice paths in \ts $\Reg(P)$ \ts
    from \ts $\zero$ \ts to \ts $(\ana,\bnb)$.
\end{lemma}

\smallskip

\begin{proof}
We first show that, for each linear extension~$L$, the corresponding
lattice path $(Z_t)_{1 \leq t \leq \n}$ is contained in \ts $\Reg(P)$.
Let \ts $t \in [n]$ \ts and let \ts $(h,k):= Z_{t}$.
Without loss of generality, we assume that $L^{-1}(t) \in \Cr_1$.
This implies that $L(\alpha_h)=t$, which in turn implies that
	\[
        L(\beta_k) \ < \ L(\alpha_h) \ < \ L(\beta_{k+1})\ts.
    \]
Now note that \. $L(\beta_k) < L(\alpha_h)$ \. implies that \ts
$\beta_k\nsucc_P\alpha_h$, and hence \.
$\big(h-\frac{1}{2}, k-\frac{1}{2}\big)\notin\Cenup(P)$.  By the same reasoning,
	we have \.  $\big(h-\frac{1}{2}, k+\frac{1}{2}\big) \notin\Cendown(P)$.
	This implies that the edge \. $\big[(h-1,k)\., \. (h,k) \big]\in \Reg(P)$.
	Since the choice of $t$ is arbitrary, this implies that
	the lattice path
	$(Z_t)_{1 \leq t \leq \n}$ is contained in $\Reg(P)$.

	We now construct the inverse map $\phi^{-1}$.
	Given a  lattice path  $(Z_t)_{1 \leq t \leq \n}$,
	 we construct the corresponding linear extension $L$ as follows.
	 For each $t \in [n]$, let
	 \begin{align*}
	 	L(\alpha_h) \ := \  t  \quad \ \text{ if } \ \quad &  Z_t-Z_{t-1} = \eone \ \ \text{ and } \ \ h=Z_t(1), \\
	 	L(\beta_k) \ := \  t  \quad \ \text{ if } \ \quad &  Z_t-Z_{t-1} = \etwo \ \ \text{ and } \ \ k=Z_t(2).
	\end{align*}
	It follows from the similar reasoning as above that $L$ respects the poset relations $\precc_P$.
	This completes the proof.
	 \end{proof}
	
	\medskip
	
		Let \. $\A=(a_1,a_2) , \B=(b_1,b_2)$ \. be two integral vertices in $\Reg(P)$,
		and let
	 $\zeta$ be a NE lattice path in $\Reg(P)$ from $\A$ to $\B$.
	Define the \defn{weight} of $\zeta$ by
	\[ \wgt(\zeta) \ := \ \text{number of unit boxes in \ts $[0,\ana] \times [0,\bnb]$ \ts that lie below $\zeta$}.  \]
	
	Recall from \eqref{eq:definition weight L} the definition of the weight function for a linear extension.
	It is easy to see that\. $\wgt(\phi(L)) \ts = \ts \wgt(L) - \binom{\ana+1}{2}$ \. for every $L\in \Ec(P)$.
 	
	\medskip
	
	\subsection{Injective maps between pairs of lattice paths}
%
	Let $\A, \B \in \Reg(P)$.
	Denote by \. $\Kc(\A,\B)$  \. the set of NE lattice paths
    $\zeta\in \Reg(P)$ that starts at $\A$ and ends at $\B$. Similarly,
    denote by \. $\aK_q(\A,\B)$  \. the polynomial
	\begin{align*}
		\aK_q(\A,\B) \  := \  \sum_{\zeta \in \Kc(\A,\B)} \. q^{\wgt(\zeta)}\..
	\end{align*}
	
	\smallskip
	
\begin{lemma}\label{l:lattice path bijection 1}
Let \ts $\A,\B \in \Reg(P)$ \ts be on the same vertical line and with
$\A$ above~$\B$, i.e.,\ \ts $a_1=b_1$ \ts and \ts $a_2 \geq b_2$.
Let \ts $\C,\D \in \Reg(P)$ \ts be on a vertical line to the right
of the line $(\A\B)$, and with $\C$ above $\D$.
	\begin{enumerate}
	\item[\textnormal{(a)}]
	If \. $|\A \B| >|\C\D|$, i.e.,\ \ts $a_2-b_2 > c_2-d_2$, then
	\begin{equation*}
				\aKr_q(\A-{\normalfont \etwo},\C) \,\cdot \, \aKr_q(\B+{\normalfont \etwo},\D) \  \geqslant \  	\aKr_q(\A,\C) \,\cdot \, \aKr_q(\B,\D).
			\end{equation*}
			\item[\textnormal{(b)}]
			If \. $|\C \D|>|\A\B|$, then
			\begin{equation*}
				\aKr_q(\A,\C-{\normalfont\etwo}) \,\cdot \, \aKr_q(\B,\D+{\normalfont\etwo}) \ \geqslant \  	\aKr_q(\A,\C)\,\cdot \, \aKr_q(\B,\D).
			\end{equation*}
		\end{enumerate}
	\end{lemma}

\medskip

\nin
Informally, the lemma says that there are more pairs of paths closer to
the inside than towards the outside of the region.  We give a direct
combinatorial proof of the lemma by an explicit injection. The injection works by translating the path $\B \to \D$ upwards so that it starts at $\A - \etwo$ and ends at $\D'$. Its translation intersects the path $\A \to \C$ and by choosing the first intersection point we can swap the paths after the intersection, creating paths $\A \to \D'$ and $\A - \etwo \to \C$. Translating the first path back, we obtain paths $\B+\etwo \to \D$ and $\A-\etwo \to \C$. We show that these paths belong to $\Reg(P)$, and the map is an injection.

	\smallskip
	
\begin{proof}
We present only the proof of part (a), as the proof of part (b) is analogous.
It suffices to show that there exists a weight-preserving injection between two set of pairs of paths
\begin{equation*}
		\vk\, : \ \Kc(\A,\C) \. \times \. \Kc(\B,\D) \quad \to \quad  	\Kc(\A-\etwo,\C) \. \times \.  \Kc(\B+\etwo,\D).
\end{equation*}
Let \ts $(\gamma,\zeta) \in \Kc(\A,\C)  \times  \Kc(\B,\D)$.
We construct a pair \ts $(\widehat{\gamma}, \widehat{\zeta}) = \vk(\gamma,\zeta)$ \ts as follows.\footnote{We suggest
the reader employ Figure~\ref{f:injection 1} as a running example.}

	\begin{figure}[hbt]
		\centering
		\includegraphics[width=0.8\linewidth]{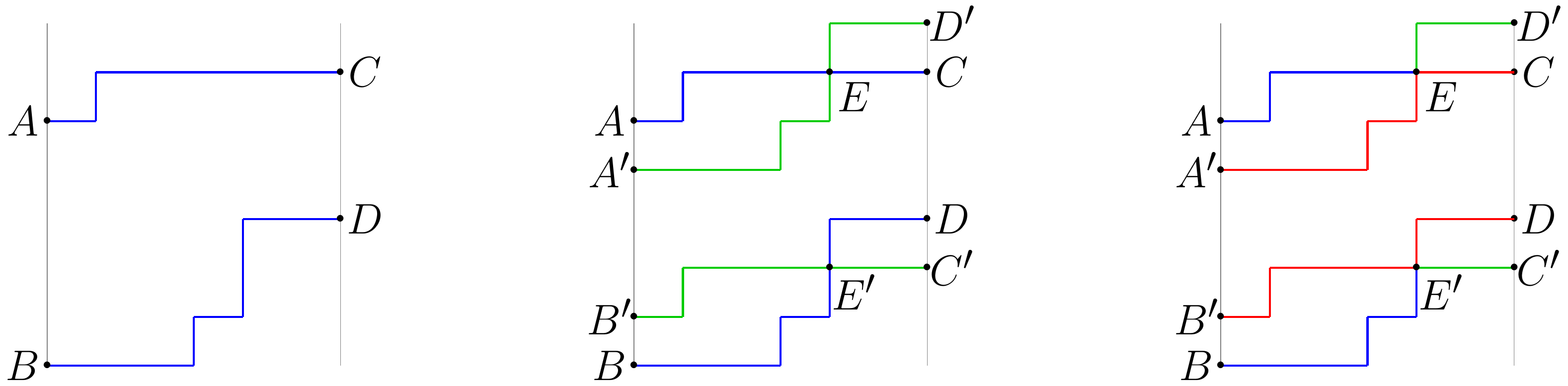}
		\caption{The lattice paths $\gamma$ and $\zeta$ (drawn in blue), the lattice paths $\gamma'$ and $\zeta'$ (drawn in green), and the lattice paths $\widehat{\gamma}$ and $\widehat{\zeta}$ (drawn in red).}
		\label{f:injection 1}
	\end{figure}

Let $\zeta'$ be the path obtained by translating  $\zeta$ by \. $(0, a_2-b_2-1)$\.,
so $\zeta'$ starts at \. $\A'=\A -\etwo$ \. and ends at \. $\D'=(d_1,d_2+a_2-b_2-1)$\..
Note that $\zeta'$ lies above~$\zeta$, that $\A'$ lies below~$\A$,
and that $\D'$ lies above $\C$, by the assumption that \. $a_2-b_2 > c_2-d_2$\..
Note also that $\zeta'$ does not necessarily belong in~$\Reg(P)$.
This implies that paths $\zeta$ and $\zeta'$ must intersect,
and let $\E$ be the first intersection point along these paths.

Let \. $\widehat{\gamma} := \zeta'(\A'\to \E) \circ \gamma(\E\to \C)$ \.
be the NE lattice path  $\A'\to \C$, such that $\widehat{\gamma}$
follows the path $\zeta': \A'\to \E$, then follows the path $\gamma:\E \to\C$.
Note that $\widehat{\gamma}\in \Reg(P)$ since both \. $\zeta'(\A' \to \E)$ \.
and \. $\gamma(\E\to \C)$ \. are contained in $\Reg(P)$.  Indeed, the former
is due to the minimality of $\E$, which implies that this portion of $\zeta'$
is below \. $\zeta\in \Reg(P)$.
	
Similarly, let $\gamma'$ be the path obtained by translating \ts $\gamma$ \ts
by \. $(0, -a_2+b_2+1)$.  Note that $\gamma'$ starts at $\B'=\B+\etwo$,
and that the first intersection point between $\gamma'$ and $\zeta$ is
\. $\E':=\E+(0,-a_2+b_2+1)$.
Let \. $\widehat{\zeta} := \gamma'(\B'\to \E') \circ \zeta(\E' \to \D)$ \.
be the NE lattice path $\B'\to \D$, such that \ts
$\widehat{\zeta}$ \ts follows the path $\gamma':\B'\to\E'$,
then follows the path $\zeta:\E'\to \D$.
Note that \. $\widehat{\zeta}\in\Reg(P)$, since  \.
$\gamma'(\B' \to \E'), \. \zeta(\E'\to \D) \in \Reg(P)$.
	
It follows from the construction above that \.
	$(\gamma',\zeta')\in \Kc(\A-\etwo,\C)  \times   \Kc(\B+\etwo,\D)$.
	This map is injective as $\gamma$ and $\zeta$ can be recovered uniquely by identifying the first intersection point $\E$.
	Furthermore, this is a weight-preserving map, since
\begin{equation}\label{eq:weight-pres}
\aligned
		 \wgt(\gamma)  \. + \.  \wgt(\zeta) \ & = \ \wgt\big(\gamma(\A \to \E) \big) \. + \.  \wgt\big(\gamma(\E \to \C) \big) \. + \. \wgt\big(\zeta(\B \to \E') \big) \. + \. \wgt\big(\zeta(\E' \to \D) \big)\\
		& = \ \wgt\big(\gamma'(\B' \to \E') \big)  \. + \.  (e_1-a_1) \times (a_2+b_2-1)  \. + \.
		\wgt\big(\gamma(\E \to \C) \big) \\
		& \hskip1.6cm + \, \wgt\big(\zeta'(\A' \to \E) \big)  \. - \.
		(e_1-a_1) \times (a_2+b_2-1)  \. + \. \wgt\big(\zeta(\E' \to \D) \big)\\
		&= \  \wgt\big(\zeta'(\A' \to \E) \big) \. + \.
		\wgt\big(\gamma(\E \to \C) \big)  \. + \.
		\wgt\big(\gamma'(\B' \to \E') \big)  \. + \.  \wgt\big(\zeta(\E' \to \D) \big)\\
		&= \ \wgt(\widehat{\gamma}) \. + \. \wgt(\widehat{\zeta})\..
\endaligned
\end{equation}	
	This completes the proof.
	\end{proof}

 \medskip

\begin{rem}\label{rem:weight-pres}
{\rm The equation~\eqref{eq:weight-pres} may seem remarkably coincidental,
but can be easily explained.  Note that when we switch paths at intersections,
the areas below paths can change but the sum of areas remain the same via
\ts $|U|+|V|=|U\cap V|+|U\cap V|$ \ts for all finite sets \ts $U,V$ \ts
of lattice squares.
}\end{rem}

\medskip

\begin{lemma}\label{l: bijection 1 equality}
	Let \ts $\A,\B, \C, \D \in \Reg(P)$ \ts be as in Lemma~\ref{l:lattice path bijection 1}.
We then have the following conditions for equalities in Lemma~\ref{l:lattice path bijection 1}:
	\begin{enumerate}
	\item[\textnormal{(a)}]
	If \. $|\A \B| >|\C\D|$, i.e.,\ \ts $a_2-b_2 > c_2-d_2$, then
	\begin{equation*}
				\aKr(\A-{\normalfont \etwo},\C) \,\cdot \, \aKr(\B+{\normalfont \etwo},\D) \  = \  	\aKr(\A,\C) \,\cdot \, \aKr(\B,\D)
			\end{equation*}
			if and only if both sides are zero, or
				\begin{equation*}
				\aKr(\A-{\normalfont \etwo},\C) \ = \ \aKr(\A,\C) \qquad \text{and} \qquad   \aKr(\B+{\normalfont \etwo},\D) \  = \  	 \aKr(\B,\D) \ = \ \aKr(\A,\D).
			\end{equation*}
			\item[\textnormal{(b)}]
			If \. $|\C \D|>|\A\B|$, then
			\begin{equation*}
				\aKr(\A,\C-{\normalfont\etwo}) \,\cdot \, \aKr(\B,\D+{\normalfont\etwo}) \ = \  	\aKr(\A,\C)\,\cdot \, \aKr(\B,\D)
			\end{equation*}
			if and only if both sides are zero, or
			\begin{equation*}
				\aKr(\A,\C-{\normalfont\etwo})   \ = \  	\aKr(\A,\C) \  = \ \aKr(\A,\D) \qquad \text{and} \qquad \aKr(\B,\D+{\normalfont\etwo})\ = \  \aKr(\B,\D)  .
			\end{equation*}			
			
		\end{enumerate}
In both cases, the equality of the number of paths implies the corresponding path collections coincide,
so the $q$-weights are also preserved.
\end{lemma}

This lemma analyzes when equality in Lemma~\ref{l:lattice path bijection 1} occurs, which is equivalent to the lattice path involution $\vk$ being a bijection. We show that unless all these paths pass vertically through points $\A$ and~$\B$, see Figure~\ref{fig:equality}, there will always be an ``extreme'' pair of paths  not contained in the image of~$\vk$.

\begin{figure}[h!]
\begin{center}
\includegraphics[width=2in]{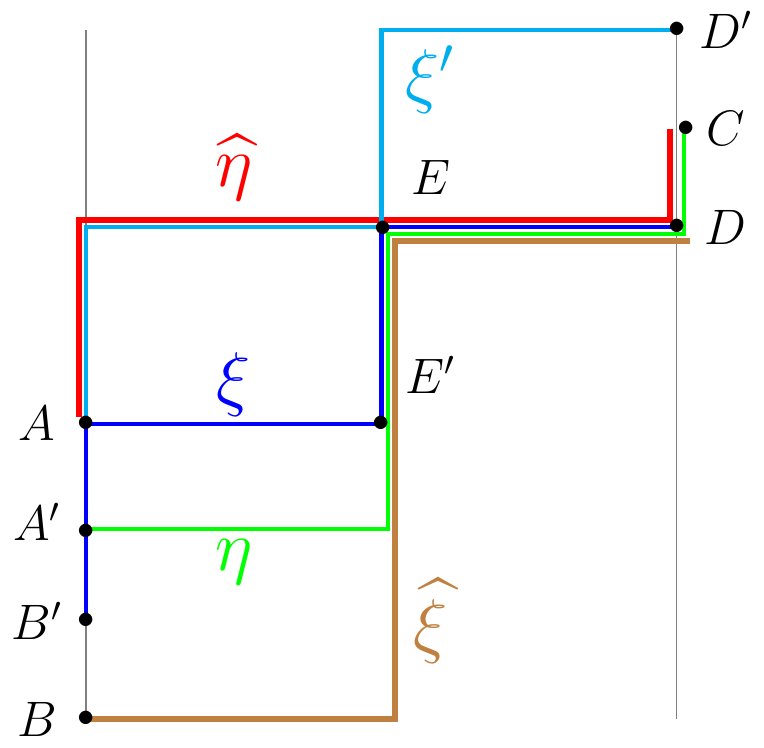}
\end{center}

\caption{The proof of the equality case, Lemma~\ref{l: bijection 1 equality}. The green path $\eta$ is the lowest in $\Reg(P)$ from $A' \to C$, and the blue path $\xi$ is the highest in the region from $B' \to D$. The cyan $\xi'=\xi+v$ is the vertical translation of $\xi$, which intersects $\eta$ at point $E$ (first such intersection). The inverse paths $\vk^{-1}(\eta,\xi) = (\wh{\eta}, \wh{\xi})$ are drawn on top in red and brown. } \label{fig:equality}
\end{figure}

\begin{proof}
We prove only part~(a), as part~(b) follows analogously.
Clearly, the ``if'' direction follows immediately.

For the ``only if'' direction, assume that the products are equal.
We will show that  \. $\aKr(\A-{\normalfont \etwo},\C) \, = \,\aKr(\A,\C)$ \. and \.
$\aKr(\B+{\normalfont \etwo},\D) \,  = \, \aKr(\B,\D)$.
From the proof of Lemma~\ref{l:lattice path bijection 1},
the equality implies that the injection $\vk$ is a bijection,
and hence surjective.

 Let $\. \eta: \A'=\A-\etwo \to \C$ \. be the {\em lowest} possible path within \ts $\Reg(P)$ \ts between these two points,
 and similarly \. $\xi:\B'=\B+\etwo \to \D$ \. be the {\em highest} possible path within \ts $\Reg(P)$ \ts between the given points, see Figure~\ref{fig:equality}.
 Let \. $\xi': \A \to \D'$ \. (which passes above the point $\C$) be the vertical translation of~$\xi$. Since $\vk$ is a bijection, we must have that $\eta$ and $\xi'$ intersect and their preimages belong to \ts $\Reg(P)$.

 First, if $d_2 \leq a'_2$, then the paths are $\xi = \B' \to (b_1',d_2) \to D$ and $\eta:A' \to (c_1,a_2')\to C$. Thus $\xi'$ lies strictly above $\eta$. Therefore, these paths do not intersect, and hence \ts $d_2>a_2'$.
 Since $\A' \in \Reg(P)$, $\A'$ is lower than $\D$ and $\xi$ is the highest path from $\B'$ to $\D$, so we must have that $\xi=\B'\to \A' \to \D$. Similarly, the lowest path must pass through $D$, so $\eta: \A' \to \D \to \C$. Furthermore, the path $\eta(\A'\to \D)$ is weakly below the path $\xi(\A' \to \D)$. Let $v=(0,a_2-b_2-1)$, the translation vector.

Since there exists a preimage \ts $\vk^{-1}(\eta,\xi)$, this implies that paths \ts $\xi'$ and $\eta$ \ts intersect.
Let $\E$ be the first intersection of \ts $\xi'$ and~$\eta$. Since \ts $\eta(\A' \to \D)$ \ts is weakly below paths \ts
$\xi$ and $\xi +v$, the point $\E$ must belong to all three paths. Then
$$\vk^{-1}\bigl(\eta,\xi\bigr) \, = \, \bigl(\ts\wh{\eta},\ts\wh{\xi}\.\bigr)\ts,
$$
where \. $\wh{\eta} = (\xi(\B'\to \E')+v) \circ \eta(\E \to \C)$ \. is a path from \ts $\A$ to $\C$,
 and  \. $\wh{\xi} = (\eta(A'\to E)-v) \circ \xi(\E' \to \D)$ \. is a path from \ts
 $\B \to \D$.
 Now note that through our assumption of $\vk$ being a bijection, we must have that $\wh{\eta}$ and $\wh{\xi}$ are both in $\Reg(P)$. Note that $\wh{\eta}$ begins with the translation of $\xi$ by $v$, and the point \. $E \in \eta(\A' \to \D) \subset \Reg(P)$. Hence \. $(\B'\to \A)\circ \wh{\eta}(\A \to \E)\circ \xi(\E \to \D) \in \Reg(P)$ \. is a path which is higher than $\xi$ in $\Reg(P)$. This causes a contradiction except in the case when $\E$ is on the line through $\A'\B'$ (and has to be equal to $\A$). This means that the lowest path \. $\eta:\A' \to \C$ \. starts with a vertical step, i.e. $\A'+\eone \not \in \Reg(P)$, and hence the lower border of \ts $\Reg(P)$ \ts contains the segment $(\B,\A)$. This implies that,  every path in $\Reg(P)$ that passes through a point in \. $\{\B,\B',\A'\}$ \. must also pass through~$\A$.

 We conclude that \. $\aK(\A,\C)=\aK(\A',\C)$, and every path \ts $\A' \to \C$ in $\Reg(P)$ passes through $\A$.  Similarly,
 we have \. $\aK(\B,\D) = \aK(\B',D) = \aK(\A,\D)$.  Finally, for the $q$-analogues we also have
 \. $\aK_q(\A,\C)=\aK_q(\A',\C)$ \. and \. $\aK_q(\B,\D) = \aK_q(\B',\D)$, since the weights are
 preserved under~$\vk$.
 \end{proof}

\medskip
			
\begin{lemma}\label{l:lattice path bijection 2}
Let \, $\A,\B \in \Reg(P)$ \. be in the same horizontal line and with $\A$ to the left of~$\B$, i.e.\ \.
$a_2=b_2$ \ts and \ts $a_1 \leq b_1$.  Let $\C,\D \in \Reg(P)$ be in a vertical line that is above the line
\ts $(\A \B)$, i.e.\ \. $c_2,d_2 \geq a_2$, and with $\C$ below~$\D$.
\begin{enumerate}
	\item[\textnormal{(a)}]  If \. $|\A \B|>0$, i.e.\ \. $b_1-a_1 > 0$, then:
			\begin{equation*}
				\aKr_q(\A+{\normalfont \eone},\C)  \,\cdot \, \aKr_q(\B-{\normalfont \eone},\D) \ \geqslant \  	\aKr_q(\A,\C)  \,\cdot \, \aKr_q(\B,\D).
			\end{equation*}
			\item[\textnormal{(b)}] 		If \.  $|\C \D|>0$,
			then:
			\begin{equation*}
				\aKr_q(\A,\C+ {\normalfont \etwo})  \,\cdot \, \aKr_q(\B,\D-{\normalfont \etwo}) \ \geqslant \  	\aKr_q(\A,\C)  \,\cdot \, \aKr_q(\B,\D).
			\end{equation*}
		\end{enumerate}
	\end{lemma}

	\smallskip
	
	\begin{proof}
		We present only the proof of part (a), as the proof of part (b) is analogous.
		It suffices to show that there exists a weight-preserving injection between two set of pairs of paths
	\begin{equation*}
		\vk \, : \, \Kc(\A,\C) \. \times \. \Kc(\B,\D) \ \to \  	\Kc(\A+\eone,\C) \. \times \.  \Kc(\B-\eone,\D).
	\end{equation*}
		Let \. $(\gamma,\zeta)\in \Kc(\A,\C) \ts \times \ts \Kc(\B,\D)$.
We construct a pair \ts $(\widehat{\gamma}, \widehat{\zeta}) = \vk(\gamma,\zeta)$ \ts as follows.\footnote{We suggest
the reader employ Figure~\ref{figinjection 2} as a running example.}

\medskip

	\begin{figure}[hbt]
	\centering
	\includegraphics[width=1\linewidth]{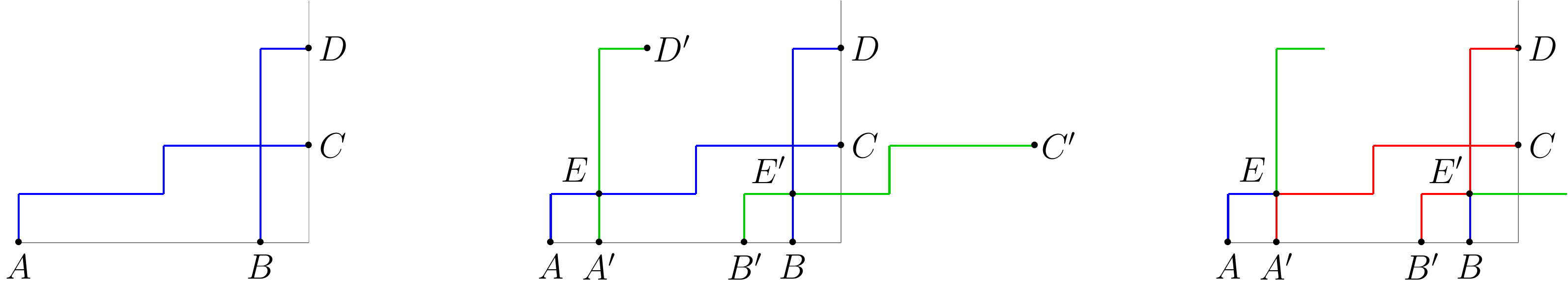}
	\caption{The lattice paths $\gamma$ and $\zeta$ (drawn in blue), the lattice paths $\gamma'$ and $\zeta'$ (drawn in green), and the lattice paths $\widehat{\gamma}$ and $\widehat{\zeta}$ (drawn in red).}
	\label{figinjection 2}
	\end{figure}
		
	Let $\zeta'$ be the path obtained by translating  $\zeta$ by \. $(a_1-b_1+1,0)$\.,
	so $\zeta'$ starts at \. $\A'=\A +(1,0)$ \.
	and ends at  \. $\D'=(d_1+a_1-b_1+1,d_2)$\..	
 	Note that $\zeta'$ lies to the left of $\zeta$,
 	that $\A'$ lies to right of $\A$,
 	and that $\D'$ lies to the left of $\D$,
 	by the assumption that \. $b_1-a_1 >0$\..
 	This implies that
 	the path $\zeta$ and $\zeta'$ must intersect,
 	and let $\E$ be the first intersection point along these paths.
 Let \. $\widehat{\gamma} := \zeta'(\A'\to \E) \circ \gamma(\E\to \C)$ \.
 be the NE lattice path from \ts $\A'\to \C$, such that \ts
 $\widehat{\gamma}$ \ts follows the path \. $\zeta': \A'\to \E$,
 then follows  the path \. $\gamma: \E\to \C$.
 	
 Let $\gamma'$ be the path obtained by translating  $\gamma$ by \. $(-a_2+b_2-1,0)$\..
 Note that $\gamma'$ starts at $\B'=\B-(1,0)$,
 and that the first intersection point between $\gamma'$ and $\zeta$ is
 \. $\E':=\E+(-a_2+b_2-1,0)$\..
 Let \. $\widehat{\zeta} := \gamma'(\B'\to \E') \. \zeta(\E', \D)$ \.
 be the NE lattice path from $\B'$ to $\D$, such that \ts
 $\widehat{\zeta}$ \ts follows the path \. $\gamma': \B'\to \E'$,
 then follows  the path \. $\zeta:\E'\to \D$.
 	
 It follows from the same argument as in the proof of Lemma~\ref{l:lattice path bijection 1},
 that~$\vk$ is an injective, weight-preserving map from
 \. $\Kc(\A,\C) \. \times \. \Kc(\B,\D)$ \. to \. $\Kc(\A+\eone,\C) \. \times \.  \Kc(\B-\eone,\D)$.
 This completes the proof.
	\end{proof}

\bigskip

\section{Lattice paths proof of Theorem~\ref{t:main-big-q}}\label{s:proof of theorem qCP}

\subsection{Setting up the injection}
We should mention that to simplify the notation, from this point on we will use
\ts $z_1 \gets x$, \ts $z_2 \gets y$, and \ts $z_3 \gets z$, and also $i \gets k$, $j \gets \ell$ in Theorem~\ref{t:main-big-q}.
By relabeling $\Cr_1$ and $\Cr_2$  and substituting $q$ with $q^{-1}$ if
necessary, we will without loss of generality assume that \ts $z_2 \in \Cr_1$.

The idea is to consider the lattice paths in $\Reg(P)$ based on the position of the horizontal step above $z_2$, which also corresponds to the value of $L(z_2)$. The summands in $\aF_q(i,j)$ correspond to lattice paths, which can be grouped according to their horizontal steps above $z_2$, say $Y \to Y + \eone$. The $i$ and $j$ give the grid distance to $Y$ from this step to the horizontal (when all $z$s are in $\Cr_1$) step above $z_1$ and $z_3$ respectively, see Figure~\ref{f:GCP}. We can expand the difference
$$\aF_q(i,j)\ \aF_q(i+1,j+1) \ -\  \aF_q(i+1,j) \ \aF_q(i,j+1)$$
as sums of pairs of lattice paths passing through the same two points above $z_2$. Then $i,i+1$ and $j,j+1$ determine which paths pass closer to each other, and we can derive the inequality by multiple applications of Lemmas~\ref{l:lattice path bijection 1} and~\ref{l:lattice path bijection 2} depending on which chains $z_1,z_3$ belong to.

Let $i,j \geq 1$ and let $Y \in \Reg(P)$.
We denote by $\Gc(i,Y)$ the set of NE lattice paths \. $\zero \to Y$ in $\Reg(P)$ \.
that pass through 	\. $Y -I $ \. and \. $Y-I+ \de_1$\., where \.
$I = I(i,Y)$ \. and \. $\de_1$ \. are defined
as
\begin{align*}
		 I \ & := \
	\begin{cases}
		(y_1-k+1, i-y_1+k-1) \ \ & \text{ if \ \, $z_1\in\Cr_1$\ts, \, and \, $z_1 = \alpha_k$\.,}\\ 			
		(i-y_2+k-1, y_2-k+1) \ \ & \text{ if \ \, $z_1\in \Cr_2$\ts, \, and \, $z_1  = \beta_k$\..}
	\end{cases}  \\
\de_1 \  & := \
	\begin{cases}
		\eone & \ \, \text{ if } \ \, z_1 \in \Cr_1\.,\\
		\etwo & \ \, \text{ if } \ \, z_1 \in \Cr_2\..
	\end{cases}
	\end{align*}

Similarly, denote by $\Hc(j,Y)$ the set of NE lattice paths \. $Y+\eone \to (\ana,\bnb)$ \.
in $\Reg(P)$ that pass through 	\. $Y +J $ \. and \. $Y +J+ \de_3$\., where \. $J =J(j,Y)$
\. and \. $\de_3$ \. are defined as
\begin{align*}
				 J \ & := \
		\begin{cases}
			(m-y_1-1, j+y_1-m+1) \ \ & \text{ if \ \, $z_3\in \Cr_1$\ts, \, and \, $z_3 \. =: \. \alpha_m$\.,}\\ 			
			(j+y_2-m+1, m-y_2-1) \ \ & \text{ if \ \, $z_3\in \Cr_2$\ts, \, and \, $z_3 \. =: \. \beta_m$\..} \\
		\end{cases}  \\
	\de_3 \ & := \
	\begin{cases}
		\eone & \ \, \text{ if } \ \, z_3 \in \Cr_1\.,\\
		\etwo & \ \, \text{ if } \ \, z_3 \in \Cr_2\..
	\end{cases}
	\end{align*}
Finally, denote	
\[ 		
\aG_q(i,Y) \ := \ \sum_{\gamma \in \Gc(i,Y)} \. q^{\wgt(\gamma)} \qquad \text{and}  \qquad
\aH_q(j,Y) \ := \ \sum_{\gamma \in \Hc(j,Y)} \. q^{\wgt(\gamma)}\..
\]

Recall  the map $\phi$ defined in the previous section.
Each linear extension $L\in \Ec(P)$ such that
	\[
    L(z_2)=u, \quad  L(z_2)-L(z_1)=i \quad \text{and} \quad L(z_3)-L(z_2)=j,
    \]
corresponds to a NE lattice path \. $(0,0)\to (\ana,\bnb)$ in $\Reg(P)$ that passes through
	\[
    \Yu- I,  \quad  \Yu-I+ \de_1, \quad \Yu, \quad \Yu+\eone, \quad \Yu+J, \quad \Yu +J+\de_3\.,
    \]
where \. $\Yu \. := \. (\ell-1,u-\ell)$\., and $\ell$ is the integer such that \ts $z_2=\al_\ell$.
That is, such a linear extension corresponds to a lattice path where the first half is contained
in \. $\Gc\bigl(i,\Yu\bigr)$ \. and the second half is contained in \. $\Hc\bigl(j,\Yu\bigr)$\..
See Figure~\ref{figlattice path decomposition} for an example.
	
\begin{figure}[hbt]
\centering
\includegraphics[width=0.25\linewidth]{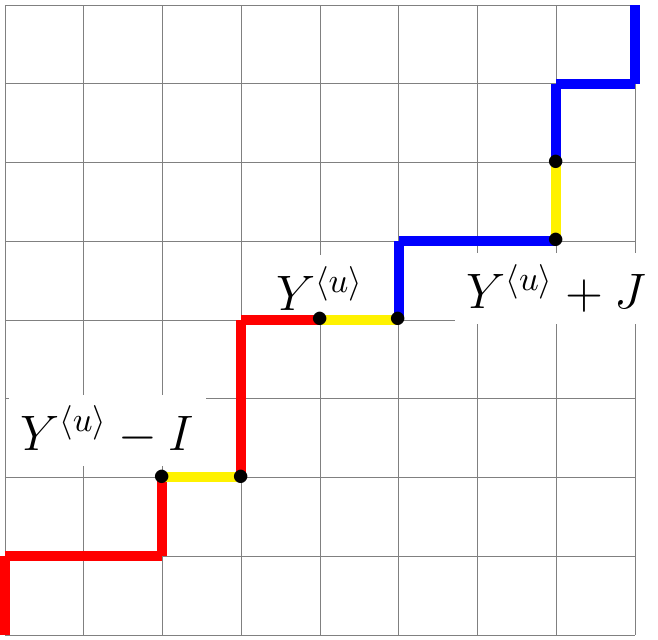}
\caption{A lattice path that corresponds to a linear extension
in $\Fc(i,j)$, with $i=j=4$  and $L(z_2)=u=8$.
Note that $z_1, z_2 \in \Cr_1$ and $z_3 \in \Cr_2$.
The first half of the lattice path $(0,0)\to (4,4)$ (in red)
is contained in \ts $\Gc\bigl(i,\Yu\bigr)$, and the second half of the path
$(5,4)\to (8,8)$ (in blue) is contained in \ts $\Hc\bigl(j,\Yu\bigr)$. }
\label{figlattice path decomposition}
\end{figure}

It now follows from the correspondence above that
	\begin{align*}
		\aF_q(i,j) \ = \ q^{\binom{\ana+1}{2}} \,  \sum_{u=1}^{\n} \, q^{u-\ell} \. \aG_q\bigl(i,\Yu\bigr) \, \aH_q\bigl(j,\Yu\bigr)\ts.
	\end{align*}
Applying the formula above  to the polynomials \. $\aF_q(i,j) \. \aF_q(i+1,j+1) $\. and \. $\aF_q(i+1,j) \. \aF_q(i,j+1)$\ts,
	we get
	\begin{equation*}
	\begin{split}
		& \aF_q(i,j) \, \aF_q(i+1,j+1)
		 \ = \ q^{\ana(\ana+1)} \,  \sum_{u=1}^{\n} \, \sum_{w=1}^{\n} \, q^{u+w-2\ell}  \. \aG_q\bigl(i,\Yu\bigr) \. \aH_q\bigl(j,\Yu\bigr) \ \aG_q\big(i+1,\Yw\big) \ \aH_q(j+1,\Yw),\\
		 & \aF_q(i+1,j) \. \aF_q(i,j+1)
		 \ = \ q^{\ana(\ana+1)}\, \sum_{u=1}^{\n}\. \sum_{w=1}^{\n} \, q^{u+w-2\ell}  \. \aG_q\big(i+1,\Yu\big) \. \aH_q(j,\Yu) \ \aG_q\big(i,\Yw\big) \ \aH_q\big(j+1,\Yw\big).
	\end{split}
	\end{equation*}
	Taking the difference between the two equation above, we get
\begin{equation}\label{eq:lattice path decomposition cross product}
\begin{split}
		& \aF_q(i,j) \. \aF_q(i+1,j+1) \. - \. \aF_q(i+1,j) \. \aF_q(i,j+1) \\
		& \hskip1.cm = \  q^{\ana(\ana+1)}\,	\sum_{u=1}^{\n} \. \sum_{w=1}^{\n} \, q^{u+w-2\ell}  \. \aH_q\big(j,\Yu\big) \   \aH_q\big(j+1,\Yw\big) \ \times \\
& \hskip3.cm \times \ \bigl[ \aG_q\big(i,\Yu\big) \ \aG_q\big(i+1,\Yw\big)  \ - \  \aG_q\big(i+1,\Yu\big) \ \aG_q\big(i,\Yw\big)  \bigr]\\
& \hskip1.cm = \ q^{\ana(\ana+1)} \, \sum_{1 \leq u <w \leq \n} \, q^{u+w-2\ell} \, \GCP_q\big(i,\Yu,\Yw\big) \, \HCP_q\big(j,\Yu,\Yw\big)\.,
\end{split}
\end{equation}
	where
	\begin{align*}
		\GCP_q(i,Y,V) \ &:= \   \aG_q(i,Y) \ \aG_q(i+1,V)  \ - \  \aG_q(i+1,Y) \ \aG_q(i,V),\\
		 \HCP_q(j,Y,V)
		 \  & := \   \aH_q(j,Y) \  \aH_q(j+1,V) \ - \
		 		\aH_q(j+1,Y) \   \aH_q(j,V).
	\end{align*}

\smallskip

Now observe that the theorem is reduced to the following result:

\smallskip

\begin{lemma}\label{l:GCP}
		Let $Y,V \in \Reg(P)$ be in the same vertical line and with $Y$ below $V$.
		Then
		\[ \GCP_q(i,Y, V) \ \geqslant \ 0 \qquad \text{ and } \qquad   \HCP_q(j,Y, V) \ \leqslant \ 0\..   \]
\end{lemma}

\medskip

\begin{proof}[{{Proof of Theorem~\ref{t:main-big-q}}}]
To obtain the theorem, apply the lemma to all the terms in~\eqref{eq:lattice path decomposition cross product}.
This gives
$$
\aF_q(i,j) \,\. \aF_q(i+1,j+1) \ \geqslant \ \aF_q(i+1,j) \,\. \aF_q(i,j+1)\ts,
$$
as desired.
\end{proof}

\medskip

\subsection{Proof of Lemma~\ref{l:GCP}}
We prove only the inequality \. $\GCP_q(i,Y, V) \. \geq \. 0$\.
as the proof of the other inequality is analogous.
		
	\begin{figure}[hbt]
	\centering
	\includegraphics[width=0.55\linewidth]{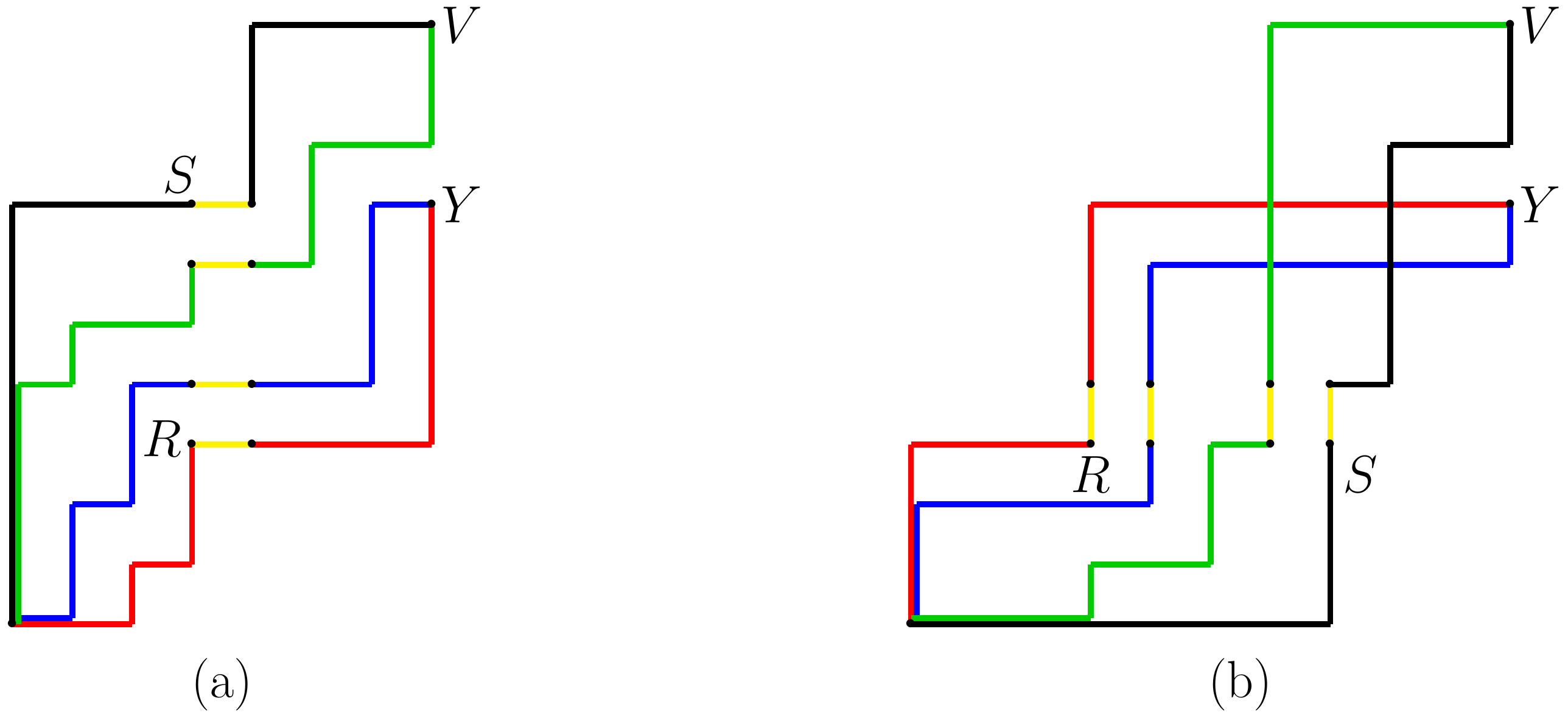}
	\caption{Two instances of lattice paths in $\Gc(i, Y)$ (in blue), $\Gc(i,V)$ (in black),  $\Gc(i+1,Y)$ (in red), and $\Gc(i+1,V)$ (in green).
	Note that $z_1 \in \Cr_1$ and $i=7$ in part (a), while $z_1 \in \Cr_2$ and $i=10$ in part (b).
	}
	\label{f:GCP}
	\end{figure}
		
\smallskip

We split the proof into two cases.
For the first case, suppose that $z_1 \in \Cr_1$.\footnote{We recommend
the reader to use Figure~\ref{f:GCP}~(a) as a running example.}
Let
		\[
            \Ga \ := \ Y \. - \. I(i+1,Y) \, =\, (\gar_1,\gar_2), \qquad \Ups \ := \ V \. -\. I(i,V) \,= \,(\ups_1,\ups_2)\ts.
        \]
        These are the points above $z_1$ where the lattice paths pass through.
In particular, the lattice paths in \. $\Gc(i,Y)$ \.
start at \ts $\zero$, ends at $Y$, and passes through \.
$\Ga+\etwo$ \. and \. $\Ga\ts +\ts\eone\ts +\ts\etwo$.
It then follows that
		\begin{align}\label{eqfoxtrot 1}
			\aG_q(i,Y) \ = \  \aK_q(\zero,\Ga+\etwo) \ q^{\gar_2+1} \ \aK_q(\Ga+\eone+\etwo,Y)\..
		\end{align}	
By an analogous reasoning, we have
				\begin{equation}\label{eq:foxtrot 2}
				\begin{split}
			\aG_q(i+1,Y) \ &= \  \aK_q(\zero,\Ga) \ q^{\gar_2} \, \aK_q(\Ga+\eone,Y)\ts,\\
			\aG_q(i,V) \ &= \  \aK_q(\zero,\Ups) \ q^{\ups_2} \, \aK_q(\Ups+\eone,V)\ts,\\
			\aG_q(i+1,V) \ &= \  \aK_q(\zero,\Ups-\etwo) \ q^{\ups_2-1} \, \aK_q(\Ups+\eone-\etwo,V)\ts.
			\end{split}
		\end{equation}
		It then follows from \eqref{eqfoxtrot 1} and \eqref{eq:foxtrot 2} that
\begin{align*}
\aG_q(i,Y) \, \aG_q(i+1,V) \  = \  q^{\gar_2+\ups_2} \,
\bigg( \aK_q(\zero,\Ups-\etwo) \, \aK_q(\zero,\Ga+\etwo) \bigg) \,
\bigg(  \aK_q(\Ups+\eone-\etwo,V) \, \aK_q(\Ga+\eone+\etwo,Y) \bigg).
\end{align*}	
We now apply Lemma~\ref{l:lattice path bijection 1}~(a) to
the last product term \. $ \aK_q(\Ups+\eone-\etwo,V) \. \aK_q(\Ga+\eone+\etwo,Y)$ \.
in the equation above, with \. $\A=\Ups+\eone$\ts, \. $\B=\Ga+\eone$\ts, \. $\C=V$ \. and \. $\D=Y$.  We get:
\begin{align*}
\aG_q(i,Y) \ \aG_q(i+1,V) \  \geqslant  \  q^{\gar_2+\ups_2} \ \bigg( \aK_q(\zero,\Ups-\etwo) \ \aK_q(\zero,\Ga+\etwo) \bigg) \, \bigg( \aK_q(\Ups+\eone,V) \ \aK_q(\Ga+\eone,Y)  \bigg).
\end{align*}	
We now apply Lemma~\ref{l:lattice path bijection 1}~(b) to first product term \.
$\aK_q(\zero,\Ups-\etwo) \. \aK_q(\zero,\Ga+\etwo)$ \. in the equation above.  We get:
with \. $\A=\B=\zero$\ts, \. $\C=\Ups$ \. and \. $\D=\Ga$,	
\begin{align*}
\aG_q(i,Y) \ \aG_q(i+1,V) \ \geqslant  \ q^{\gar_2+\ups_2} \ \bigg( \aK_q(\zero,\Ups) \ \aK_q(\zero,\Ga)  \bigg) \, \bigg(  \aK_q(\Ups+\eone,V) \ \aK_q(\Ga+\eone,Y)  \bigg).
\end{align*}	
		It then follows from \eqref{eq:foxtrot 2} that
\begin{align*}
\aG_q(i,Y) \ \aG_q(i+1,V) \  \geqslant \   \aG_q(i+1,Y) \ \aG_q(i,V)\ts.
\end{align*}
		This proves that \. $\GCP_q(i,Y, V) \. \geqslant \. 0$ \. for the first case.
		
		\smallskip
		
For the second case, suppose that $z_1 \in \Cr_2$.\footnote{We recommend
the reader to use Figure~\ref{f:GCP}~(b) as a running example.}
We write
		\[
                \Ga \ := \ Y\. -\. I(i,Y), \qquad \Ups \ := \ V \. -\. I(i+1,V)\..
        \]
		It then follows that
		\begin{equation}\label{eq:foxtrot 3}
		\begin{split}
			\aG_q(i,Y) \ &= \ \aK_q(\zero, \Ga+\eone) \ \aK_q(\Ga+\eone+\etwo,Y)\ts,\\
			\aG_q(i+1,Y) \ &= \ \aK_q(\zero, \Ga) \ \aK_q(\Ga+\etwo,Y)\ts,\\
			\aG_q(i,V) \ &= \ \aK_q(\zero, \Ups) \ \aK_q(\Ups+\etwo,V)\ts,\\
			\aG_q(i+1,V) \ &= \ \aK_q(\zero, \Ups-\eone) \ \aK_q(\Ups-\eone+\etwo,V)\ts.
		\end{split}
		\end{equation}
		It then follows from \eqref{eq:foxtrot 3} that
		\begin{align*}
			\aG_q(i,Y) \ \aG_q(i+1,V) \ = \  \bigg( \aK_q(\zero,\Ga+\eone) \
\aK_q(\zero,\Ups-\eone)  \bigg) \ \bigg(\aK_q(\Ga+\eone+\etwo,Y) \ \aK_q(\Ups-\eone+\etwo,V)\bigg)\..
		\end{align*}	
		We now apply Lemma~\ref{l:lattice path bijection 2}~(a) to
		the second product  term \. $\aK_q(\Ga+\eone+\etwo,Y) \. \aK_q(\Ups-\eone+\etwo,V)$ \. in the equation above,
		with \. $\A=\Ga+\etwo$\ts, \. $\B=\Ups+\etwo$\ts, \. $\C=Y$, \. $\D=V$,	
		\begin{align*}
		\aG_q(i,Y) \ \aG_q(i+1,V) \  \geqslant \  \bigg( \aK_q(\zero,\Ga+\eone)
        \ \aK_q(\zero,\Ups-\eone)  \bigg) \ \bigg(  \aK_q(\Ga+\etwo,Y) \ \aK_q(\Ups+\etwo,V)   \bigg)\..
		\end{align*}	
		We now apply Lemma~\ref{l:lattice path bijection 2}~(b) to
		the first product term \. $\aK_q(\zero,\Ga+\eone) \. \aK_q(\zero,\Ups-\eone) $ \. in the equation above,
		with \. $\A=\B=\zero$, \. $\C=\Ga$, \. $\D=\Ups$,	
\begin{align*}
\aG_q(i,Y) \ \aG_q(i+1,V) \  \geqslant \   \bigg( \aK_q(\zero,\Ga) \ \aK_q(\zero,\Ups)  \bigg) \ \bigg(  \aK_q(\Ga+\etwo,Y) \ \aK_q(\Ups+\etwo,V)   \bigg)\..
		\end{align*}	
				It then follows from \eqref{eq:foxtrot 3} that
		\begin{align*}
			\aG_q(i,Y) \ \aG_q(i+1,V) \  \geqslant \   \aG_q(i+1,Y) \ \aG_q(i,V)\ts.
		\end{align*}
This proves that \. $\GCP_q(i,Y, V) \. \geq \. 0$ \. for the second case,
and our proof is complete. \qed

\bigskip
	
\section{Proof of Theorem~\ref{t:cpc_equality}} \label{s:equality}
The cross-product equality is obtained by analyzing the proof in Section~\ref{s:proof of theorem qCP}
and applying Lemma~\ref{l: bijection 1 equality}.
We consider only the case when $z_1,z_2,z_3\in \Cr_1$, as the other cases are analogous.

\smallskip
Clearly, we have \. {\bf (a), (b), (c)}~$\Rightarrow$~\eqref{eq:cpc_equality}.
If \. {\bf (d)} holds, this implies that \ts $P = P_1 \cup P_2$, where \.
$P_1 := \{ x \, : \, x \preccurlyeq_P z_2 \}$ \. and \. $P_2 := \{y \, : \, y \succ_P z_2 \}$\., and every element in $P_1$ is  smaller than every element in $P_2$ by $\prec_P$.
Then \. $\aF(i,j) = \aF'(i) \cdot \aF''(j)$, where $\aF'(i)$ is the number of linear extensions $L_1$ of $P_1$ \. s.t.\ \. $L_1(z_2) - L_1(z_1)=i$.
 Similarly, $\aF''(j)$ \ts is the number of linear extensions $L_2$ of $P_2$ \. s.t.\ \. $L_2(z_3)=j-1$. Then:
$$
\aF(i,j) \ \aF(i+1,j+1) \ = \ \aF'(i) \ \aF'(i+1) \ \aF''(j) \ \aF''(j+1) \ = \  \aF(i+1,j) \ \aF(i,j+1).
$$

We now prove \eqref{eq:cpc_equality}~$\Rightarrow$~\textbf{(a)},~\textbf{(b)},~\textbf{(c)},~or~\textbf{(d)}.
Suppose now that equation~\eqref{eq:cpc_equality} holds. Suppose that \. $\aF(i+1,j)\ \aF(i,j+1) >0$. Since \ts $\aF(i+1,j) >0$, then there is at least one linear extension $L$ of $P$, such that $L(z_2)-L(z_1)=i+1$ and $L(z_3)-L(z_2)=j$, and consider the one for which $L(z_2)=w$ is maximal. Hence \. $\aG(i+1, Y) \, \aH(j, Y) >0$ \. for \ts $Y = \Yw$, and \, $\aG(i+1, Y) \, \aH(j, Y) =0$ \. for \ts $Y$ higher than~$\Yw$.  Let $u$ be the minimal value for which \. $\aG(i,\Yu) \, \aH\big(j+1,\Yu\big) >0$.

Let us show that, if $L$ is a linear extension in one of the sets \. $\Fc(i,j)$, \. $\Fc(i+1,j)$,
\. $\Fc(i,j+1)$, or $\Fc(i+1,j+1)$,	then \. $u \leq L(z_2) \leq w$.
	Formally, we will prove that if \ts $t <u$, \ts then  \.
$\aG\bigl(i+1,\Yt\bigr) \, \aH\bigl(j,\Yt\bigr)=0$; the other cases are analogous.
	Suppose to the contrary, that \. $\aG\bigl(i+1,\Yt\bigr) \, \aH\bigl(j,\Yt\bigr) >0$.
	Then there is a path in \ts $\Reg(P)$ \ts that goes through \. $\zero$, \ts $A:=\Yu-I(i,\Yu)$, \ts $B:=A+\eone$ \ts and~$\ts \Yu$.
Similarly, there is a path in \ts $\Reg(P)$ \ts that goes through $\zero$, $C:=\Yt-I(i+1,\Yt)$, $D:=C+\eone$ \ts and~$\ts \Yt$.

Since $\Reg(P)$ is the region between two monotonous NE paths, it contains the segment \ts $\bigl(\Yt,\Yu\bigr)$ \ts and its $\eone$ translate, and similarly the segment $AC$ and its $\eone$ translate. Thus we can take $t=u-1$. Then the points $Y^{\<u-1\>} - I\bigl(i,Y^{\<u-1\>}\bigr)$, \ts $Y^{\<u-1\>} - I\bigl(i,Y^{\<u-1\>}\bigr) +\eone$, \ts $Y^{\<u-1\>}$, \ts$ Y^{\<u-1\>} + \eone$ are in $\Reg(P)$ and so $\aG\bigl( i, \Yt \bigr) >0$. Similarly, since $\Yu  + J\bigl(j,\Yu\bigr) = Y^{\< u-1\>}  +   J\bigl(j+1,Y^{\<u-1\>}\bigl) \in \Reg(P)$ we have that $\aH\bigl(j+1, \Yt \bigr) >0$. This contradicts the minimality of~$u$ and completes the proof of this claim.

\smallskip

There are now two cases.
Suppose first that $u<w$, then we have from the proof of Theorem~\ref{t:main-big-q}, and in particular equation~\eqref{eq:lattice path decomposition cross product} and Lemma~\ref{l:GCP}, that
$$
\aF(i,j+1) \ \aF(i+1,j) \ - \ \aF(i,j)\ \aF(i+1,j+1) \ \geq  \ - \, \GCP\big(i,\Yu,\Yw\big) \, \HCP\big(j,\Yu, \Yw\big)\.,
$$
since on the RHS, we have \. $\GCP(\cdot) \geq  0$ \. and \. $\HCP(\cdot) \leq 0$.
Since by \eqref{eq:cpc_equality} the LHS is equal to zero, we
must have \. $\GCP\big(i,\Yu,\Yw\big)=0$ \. or \. $\HCP\big(j, \Yu, \Yw\big)=0$.

We now show that $\GCP(i,\Yu,\Yw)=0$ leads to \textbf{(a)}.
Let $S$ be the point above $z_1$, such that the grid distance between $S$ and $\Yw$ is equal to~$i$.
In other words, define \. $S:=\Yw-I\big(i,\Yw\big)$.  Then every linear extension $L\in \Ec(P)$ for which \.
$L(z_2)-L(z_1)=i$ \. and \. $L(z_2)=w$, corresponds to a path which passes through the segment \ts
$\bigl(S,S+\eone\bigr)$.
Similarly, let $R$ be the point above $z_1$ at grid distance $(i+1)$ from $\Yu$, and let \. $R:=\Yu-I\big(i+1,\Yu\big)$.
See Figure~\ref{f:GCP}~(a), where $V=\Yw$ and~$Y=\Yu$.

Denote by \ts $M_1$ \ts the number of pairs \ts $(\zeta,\gamma)$ \ts
of paths \. $\zeta: \zero \to S \to (S+\eone) \to \Yw$ \. and \. $\gamma: \zero \to R \to (R+\eone) \to \Yu$ in $\Reg(P)$.
Similarly, denote by \ts $M_2$ \ts the number of pairs \ts $(\zeta',\gamma')$ \ts of paths \.
$\zeta': \zero \to (S-\etwo) \to (S-\etwo+\eone) \to \Yw$ \. and \. $\gamma': \zero \to (R+\etwo)\to (R+\etwo+\eone) \to \Yu$ in $\Reg(P)$.
Then \. $\GCP(i,\Yu,\Yw) =0$ \. is equivalent to \. $M_1 = M_2$.


 On the other hand, by Lemma~\ref{l:lattice path bijection 1}, we have:
$$
\aK\bigl(S+\eone, \Yw\bigr) \ \aK\bigl( R+\eone, \Yu\bigr) \ \leq \ \aK\bigr(S-\etwo +\eone, \Yw\bigr) \ \aK\big( R+\etwo + \eone, \Yu\big)
$$
and
$$
\aK(\zero, S)\ \aK(\zero, R) \ \leq \ \aK(\zero, S-\etwo) \ \aK(\zero, R+\etwo).
$$
Since $\GCP(i, \Yu,\Yw)=0$ both of these inequalities have to be equalities. We now apply Lemma~\ref{l: bijection 1 equality} for these two cases (paths starting at $\zero$, and paths ending at $\Yu$ and~$\Yw$) and its analysis on the possible paths in case of equality. It implies that, for every $T$
in the segment $SR$,
 all the paths \. $(T+\eone) \to \Yu$ \. in $\Reg(P)$ must pass through \ts $(S + \eone)$.  Similarly, all paths \. $\zero \to T$ \. in $\Reg(P)$ must pass through~$R$.
  Thus, for all $t$ satisfying $u \leq t \leq w$,
  we have  \. $\aG\bigl(i,\Yt\bigr) \. = \. \aG\bigl(i+1,\Yt\bigr)$.  In other words,
  the number of paths \. $\zero \to \Yt$ \. in \ts $\Reg(P)$ \ts passing through point \ts $T:=\Yt-I(i,\Yt)$ and $T+\eone$,
  is equal to the  number of paths passing through \. $T-\etwo =\Yt-I(i+1,\Yt)$ \. and \. $T-\etwo+\eone$.


%
This implies:
\begin{alignat*}{3}
		\aF(i,j) \ & = \  \sum_{u \leq t \leq w} \aG\bigl(i,\Yt\bigr) \, \aH\bigl(j,\Yt\bigr) \ &&= \ \sum_{u \leq t \leq w} \aG\bigl(i+1,\Yt\bigr) \, \aH\bigl(j,\Yt\bigr)  \ && = \ \aF(i+1,j),\\
		\aF(i,j+1) \ & = \  \sum_{u \leq t \leq w} \aG\bigl(i,\Yt\bigr) \, \aH\bigl(j+1,\Yt\bigr) \ &&= \ \sum_{u \leq t \leq w} \aG\bigl(i+1,\Yt\bigr) \, \aH\bigl(j+1,\Yt\bigr)  \ && = \ \aF(i+1,j+1),
\end{alignat*}
 which leads to case {\bf (a)}.
The case \. $\HCP\big(j,\Yu,\Yw\big)=0$ \. similarly leads to {\bf (b)}.

We now show that the case $u=w$ lead to case \textbf{(d)}.
Suppose to the contrary that there exists $t \neq u$ such that $L(z_2)=t$, for some \ts $L\in \Ec(P)$.
Suppose  $t < u$, the case $t>u$ follows analogously. Then $\Yt$ is in $\Reg(P)$.
By the geometry of \. $\Reg(P)$, the segment \. $\bigl(\Yt,\Yu\bigr) $ \. is also in $\Reg(P)$,
and so \. $Y^{\<u-1\>}$ \. is in  $\Reg(P)$.  Since \. $\aG\big(i+1,\Yu\big)>0$, we have points
$$
R \, = \, \Yu - I\big(i+1,\Yu\big) \, = \, Y^{\<u-1\>}-I\big(i,Y^{<u-1>}\big) \ \text{ and }  \ R+\eone
$$
both contained in \ts $\Reg(P)$.

Since the boundaries of \ts $\Reg(P)$ \ts are NE paths, there must be a path \. $\zero \to  R \to   (R+\eone) \to Y^{\<u-1\>}$ \. in \ts $\Reg(P)$.
Similarly, on the other side, there is a path \. $\bigl(Y^{\<u-1\>} +\etwo\bigr) \to \bigl(\Yu + J\bigr) \to \big(\Yu+J+\eone\big) \to Q$, where $J:=J\big(j,\Yu\big)$. Thus we have \. $\aG\big(i,Y^{\<u-1\>}\big) >0$ \. and \. $\aH\big(j+1,Y^{\<u-1\>}\big)>0$, contradicting the minimality of~$u$.  This completes the proof of the first part of the theorem.

\smallskip

For the second part, we clearly have \eqref{eq:cpc_equality-q} implies~\eqref{eq:cpc_equality} by setting \ts $q=1$.
In the opposite direction, the first part states that either of \ts {\bf (a)}--{\bf (d)} \ts holds.  In case
\ts {\bf (c)} \ts both sides are zero, and in case \ts {\bf (d)} \ts equality~\eqref{eq:cpc_equality-q}
follows immediately since both sides give a $q$-counting of the same family of quadruples of paths.  In case
\ts {\bf (a)}, we have \. $\aG_q\bigl(i,\Yt\bigr) = q \ts \aG_q\bigl(i+1,\Yt\bigr)$ for every~$t$, and
the above calculation gives
$$
\aF_q(i,j) \ = \ q\. \aF_q(i+1,j) \quad \text{and} \quad \aF_q(i,j+1) \ = \ q\. \aF_q(i+1,j+1)\ts.
$$
This gives~\eqref{eq:cpc_equality-q} as the $q$-terms cancel.  Finally, the case  \ts {\bf (b)} \ts
is analogous to \ts {\bf (a)}. This completes the proof of the second part of the theorem.   \qed

\bigskip

\section{Final remarks and open problems}\label{s:finrem}

\subsection{} \label{ss:finrem-hist}
The number $e(P)=|\Ec(P)|$ of linear extensions was shown to be $\SP$-complete for
general posets by Brightwell and Winkler~\cite{BW}.  Recently, it was shown to be
$\SP$-complete for \emph{dimension two posets}, \emph{height two posets}, and for
\emph{incidence posets}.  In the opposite directions, there are several classes of
posets where computing $e(P)$ can be done in polynomial time, see a historical
overview in~\cite{DP}.  Note that in contrast to many other $\SP$-complete problems,
the decision problem \. $e(P)>^{?}0$ \. is trivial, and that $e(P)$ has a polynomial time
$(1\pm \ve)$ approximation (ibid.) This make the problem most similar to the
\ts {\sf BINARY PERMANENT}, where the decision problem is classically in~{\sc P}.

\subsection{} \label{ss:finrem-1323}
The \ts $\frac{1}{3}-\frac{2}{3}$~Conjecture~\ref{conj:1323} \ts was posed
independently by Kislitsyn~\cite{Kis} and Fredman~\cite{Fre} in the context
of sorting.  The currently best general bounds are obtained in~\cite{BFT}, which
both used and extended the arguments in~\cite{KS}.  As mentioned in the introduction,
the author's main lemma is the proof of the Cross--Product Conjecture~\ref{conj:CP}
for special values \ts $k=\ell=1$.

Note that there seem to be evidence that the \ts $\frac{1}{3}-\frac{2}{3}$~conjecture \ts
is unattainable by means of general poset inequalities, see a discussion
in~\cite[p.~334]{BFT}.  In a different direction, much effort has been made
to resolve the conjecture in special cases, see e.g.\ \cite[$\S$1.3]{CPP1}
for a recent overview.

We should also mention that the constant \ts $\frac13$ \ts is tight for a $3$-element
poset, but is likely not tight for many classes of posets such as posets of larger
width and indecomposable posets. Notably, there is a robust recent literature
on getting better bounds for posets of width two, see e.g.~\cite{Chen,Sah}.

\subsection{} \label{ss:finrem-CPC}
When stating CPC in~\cite{BFT}, the authors were explicitly motivated by~\cite{KS},
but they did not seem to realize that CPC easily implies the Kahn--Saks
Theorem~\ref{t:KS}.  This implication is described in~$\S$\ref{ss:power-KS}.
Note that it increases the width of the poset, so our proof of CPC for width two posets
is by itself inapplicable.

The implication above suggests that in full generality, perhaps one should
look for a geometric proof of CPC rather than refine combinatorial arguments.
Indeed, as of now, there is no combinatorial proof of the Kahn--Saks
inequality~\eqref{eq:KS-ineq} in full generality. If anything, the passage
of time since the powerful  FKG and XYZ inequalities were discovered
(see e.g.~\cite[Ch.~6]{AS}), suggests that  inequalities
such as CPC are fundamentally harder in their nature (cf.~\cite{Pak}).

Perhaps, this can be explained by the FKG and XYZ inequalities being
in the family of \emph{correlation inequalities} (inequalities involving
only relations $x\prec y$), while the Kahn--Saks and cross--product
inequalities being in the family of \emph{coordinate-wise
inequalities} (inequalities involving the relations \ts $L(y)-L(x)=i$).
In other words, the latter involve finer statistics of linear extensions.

Finally, CPC is closely related to the \emph{Rayleigh property} which
plays an important role in the study of \emph{negative dependence} in
probability and combinatorics, see~\cite{BBL,BH}.  We also refer to~\cite{Huh}
for a recent broad survey of such quadratic inequalities
from algebraic and geometric points of view.

\subsection{} \label{ss:finrem-KS}
In a forthcoming paper~\cite{CPP2}, we derive the $q$-analogue of the Kahn--Saks
inequality~\eqref{eq:KS-ineq} for width two posets using the lattice paths
approach.  Formally, let
\begin{equation}\label{eq:def-weight-q-analogue-KS}
\aF_q(k) \ := \  \sum_{L}  \, q^{\wgt(L)}\.,
\end{equation}
where the summation is over all linear extensions \ts $L\in \Ec(P)$,
such that \ts $L(y)-L(x)=k$.

\smallskip

\begin{thm}[{\rm\defn{$q$-Kahn--Saks inequality}~\cite{CPP2}}]\label{t:KS-q}
Let \. $P=(X,\prec)$ \. be a finite poset of width two, let $(\Cr_1,\Cr_2)$
be a partition of~$P$ into two chains.  For all distinct elements \ts $x,y,z\in X$,
we have:
\begin{equation}\label{eq:KS-ineq-q}
\aFr_q(k)^2 \ \geqslant \ \aFr_q(k-1) \ \aFr_q(k+1) \quad \text{for all} \quad k\. > \. 1\ts,
\end{equation}
where \ts $\aFr_q(k)$ \ts is defined in~\eqref{eq:def-weight-q-analogue-KS}, and
the inequality between polynomials is coefficient-wise.\end{thm}

\smallskip

Note that the inequality~\eqref{eq:KS-ineq-q} does not seem to follow from
our $q$-analogue~\eqref{eq:CP-ineq-q} of the cross--product inequality,
because of the width increase described in~$\S$\ref{ss:finrem-CPC}.
Nor does~\eqref{eq:KS-ineq-q}  seem to follow from
geometric techniques in~\cite{KS,Sta}. While the tools involved in the proof
of Theorem~\ref{t:KS-q} are somewhat similar to the tools in this paper, the
details are surprisingly intricate and goes beyond the scope of this paper.

\subsection{} \label{ss:finrem-total}
There are classical connections between \emph{log-concavity} and
\emph{total positivity}, see e.g.~\cite{Bre}. Our property of
nonnegative \ts $2\times 2$ \ts minors is similar but weaker
than the total nonnegativity.  There are two reasons for us using
this weaker property: practical and technical.  On the one hand,
the \ts $2\times 2$ \ts minors suffice for our purposes, while
signs of large size minors does not seem to follow from our
analysis of admissible vectors in Section~\ref{s:cross product relation}.

Initially we believed that our algebraic approach points towards total
nonnegativity of matrix \ts $\bF_P^{\ts \vee}$ \ts obtained from \ts
$\bF_P$ \ts by reversing the order of the second index.  A counterexample
to this natural conjecture was recently found by Jacob B.~Zhang by computer
experiments.\footnote{Personal
communication (May 5, 2021). }  Compare this with
Lemma~\ref{l:characteristic matrix}, which implies that the characteristic
matrix~$\bN_P$ is \emph{totally nonnegative} for all posets $P$ of
width two.

\subsection{} \label{ss:finrem-BFT}
The fundamental idea of splitting linear extensions $\Ec(P)$
into two parts is one common feature of the proof in~\cite{BFT} and
both our proofs (see Sections~\ref{s:proof of CPC} and~\ref{s:proof of theorem qCP}).
Curiously, in~\cite[p.~338]{BFT} the authors suggest that the case \ts $k=\ell=1$
\ts of the Generalized Cross--Product Conjecture~\ref{conj:CP-gen} can be obtained
by their methods.  We also believe this to be the case.  It would be interesting
to see if this approach can be utilized to derive other results, perhaps beyond
the cross--product inequality framework.

\subsection{} \label{ss:finrem-GenCPC}
A casual reader might conclude that Cross--Product Conjecture~\ref{conj:CP}
implies  Generalized Cross--Product Conjecture~\ref{conj:CP-gen} by the following
argument:  write \ts $\aF(k+i,m)/\aF(k,m)$ \ts as a telescoping product and apply
the CPC to the factors shows that the ratio is non-increasing in~$m$, giving the
GCPC.  This would be true if it was clear that all the factors are nonzero.
As it happens, determining when \ts $\aF(i,j)=0$ \ts is rather difficult; see
\cite[$\S$8]{CPP2} where the Kahn--Saks inequality case of \ts $\aF(i)=0$ \ts 
was resolved. We intend to pursue this direction in the forthcoming paper~\cite{CPP3}.

\subsection{} \label{ss:finrem-q}
The number \ts $e(P)=|\Ec(P)|$ \ts of linear extensions already has a notable
$q$-analogue generalizing \emph{major index} of permutations.  This was introduced
by Stanley, see e.g.~\cite[$\S$3.15]{EC} and~\cite{KimS} for a more recent references.
Note that this $q$-analogue depends only on the poset~$P$, even though the underlying
statistics depends on the fixed linear extension $L\in \Ec(P)$. On the other hand,
the $q$-analogue \ts $\aF_q(k,\ell)$ \ts defined in the introduction,
depends on the chain partition \ts $(\Cr_1,\Cr_2)$ \ts as a polynomial.

\subsection{} \label{ss:finrem-GYY}
The Graham--Yao--Yao (GYY) inequality~\eqref{eq:GYY} is less known than the other
poset inequalities in this paper, and can be viewed as an ultimate positive correlation
inequality for posets of width two.  Curiously, the original proof also
used lattice paths; the authors acknowledged Knuth for simplifying it.
A different proof using the powerful FKG inequality was given by
Shepp~\cite{She}.  It would be interesting to see if another Shepp's
inequality \cite[Thm~2]{She} can also be derived from the CPC.

\subsection{} \label{ss:finrem-equality}
The equality part for the Stanley inequality~\eqref{eq:stanley} was recently
characterized in~\cite[Thm~15.3]{SvH} for all posets, as an application a
difficult geometric argument.  In notation of Corollary~\ref{c:logconcavity-Stanley},
they show that \. $\qbr_x(i)^2  \. = \. \qbr_x(i-1) \. \qbr_x(i+1)$ \.
if and only if \. $\qbr_x(i-1) = \qbr_x(i)= \qbr_x(i+1)$.  In~\cite{CP},
the first two authors extend this result to weighted linear extensions, 
but the weights there are quite different from the weights in~\eqref{eq:def-weight}. 

In~\cite{CPP2}, we extend the above equality conditions of Stanley's inequality 
to the equality conditions of the Kahn--Saks inequality~\eqref{eq:KS-ineq}, 
but only in a special case.  In the notation of Theorem~\ref{t:KS}, we prove that the equality \. 
$\aF(k)^2 \. = \. \aF(k-1) \. \aF(k+1)$ \. 
implies \. $\aF(k-1) \. = \. \aF(k) \. = \. \aF(k+1)$ \. 
for posets~$P$ of width two, 
and when elements \. $x,y\in X$ \. belong to the same chain in a partition 
of $P$ into two chains \ts $(\Cr_1, \Cr_2)$.
As we mentioned above, this result
does not follow from our Cross--Product Equality Theorem~\ref{t:cpc_equality}.

On the other hand, Theorem~\ref{t:cpc_equality}
does not hold for all posets. Indeed, let \. $P=C_m+C_m+ C_1$
be the disjoint sum of three chains of size \ts $m$, \ts $m$ \ts and~$\ts 1$, respectively,  
where \ts $m\ge 3$.  Denote these chains by
$\Cr_1:=\{\al_1,\ldots, \al_m\}$, \.
$\Cr_2:=\{\beta_1,\ldots, \beta_m\}$, \. $\Cr_3:=\{\gamma\}$.
Let \. $x:=\al_1$, \. $y:=\gamma$, \. $z:=\beta_m$.
 Then we have:
 \[  \aF(i,j) \ = \ 2^{i+j-2} \qquad \text{ for all  } \ i,j \geq 1 \ \text{ and } \ i+j \leq m+1.  
 \]
Fix \ts $k,\ell \geq 1$ \ts such that \ts $k+\ell \le m-1$. The cross--product 
equality~\eqref{eq:cpc_equality} holds in this case:
\[  
\aF(k,\ell) \ \aF(k+1,\ell+1) \ = \  \aF(k+1,\ell) \ \aF(k,\ell+1) \ = \ 2^{2k+2\ell-2},  
\]
 but neither of the conditions \. {\small \bf (a)}, {\small \bf (b)}, {\small \bf (c)}, 
 {\small \bf (d)} \. in Theorem~\ref{t:cpc_equality} applies.

Let us also mention that by using the argument in $\S$\ref{ss:power-KS},
one can transform the example above into an equality case of Kahn--Saks inequality~\eqref{eq:KS-ineq}
for which \. $\aF(k)^2= \aF(k+1) \aF(k+1)$ \. but $\aF(k) \neq \aF(k+1) \neq \aF(k-1)$, 
see \cite[Ex.~1.5]{CPP2} for further details.

%
%

\subsection{} \label{ss:finrem-XYZ}
The remarkable $XYZ$ inequality~\eqref{eq:XYZ} was originally conjectured by
Rival and Sands (1981) and soon after proved by Shepp~\cite{She-XYZ} by a
delicate use of the FKG inequality.  To quote the original paper, the
$XYZ$ inequality is ``surprisingly difficult to prove in spite of much
effort by combinatorialists'' (ibid.)  Winkler shows in~\cite{Win1} that
in some sense all correlation inequalities of a certain type must follow
from the $XYZ$ inequality.

Curiously, when $x,y$ and~$z$ form an antichain, the $XYZ$
inequality~\eqref{eq:XYZ} is always strict. This was proved by
Fishburn~\cite{Fish} with an explicit lower bound on the ratio.
Applying this result to \eqref{eq:xyzdifference} for posets of width two,
we see that the sum in the right side of \eqref{eq:xyzdifference}
is always strictly positive.  This implies that one can always find \ts
$i,j<0$ \ts and \ts $k,\ell>0$, such that the
inequality~\eqref{eq:CP-ineq-gen-even-more} is strict.
For example, for  the poset \ts $P=C_m+C_m+C_1$ \ts as above,
strict inequality occurs for \. $i=j=-1$ and $k=\ell=2$, since $\aF(i,j)=0$ and $\aF(i,\ell), \aF(k,j)>0$. 
Note that this is not the only instance of strict inequalities in this example.

Finally, let us mention~\cite{BT} which shows the difficulty of the
equality problem in a small special case of a related problem.  
We also refer to a somewhat dated survey~\cite{Win2},
where Winkler emphasizes the importance of finding strict inequalities.

\vskip.7cm
	
\subsection*{Acknowledgements}
We are grateful to June Huh, Jeff Kahn and Nati Linial for interesting conversations.  
We thank Yair Shenfeld and Ramon van Handel for telling us about~\cite{SvH}, and for
helpful remarks on the subject. Additionally, we thank Ramon van Handel and Alan Yan
for the Example~1.5 in~\cite{CPP2}, which inspired our counterexample 
in~$\S$\ref{ss:finrem-equality}. 
Special thanks to Tom Trotter for suggesting we look into the cross--product
conjecture.  
We also thank Jacob~B.~Zhang for careful reading of the paper and for providing
the counterexample in \S\ref{ss:finrem-total}.
 The last two authors were partially supported by the NSF.

\vskip1.1cm


%

\

\vskip.7cm

\end{document}